\documentclass[a4paper,10pt]{amsart}
\usepackage[matrix,arrow,tips,curve]{xy}
\usepackage[english]{babel}
\usepackage[leqno]{amsmath}
\usepackage{amssymb,amsthm}
\usepackage{amscd}
\usepackage{enumerate}

\input{xy}
\xyoption{all}

\newtheorem{thm}{Theorem}[subsection]
\newtheorem{lemma}[thm]{Lemma}

\newtheorem{prop}[thm]{Proposition}

\newtheorem{assumption}[]{Assumption}
\newtheorem{cor}[thm]{Corollary}

\newtheorem{conj}[thm]{Conjecture}
\newtheorem{prob}[]{Problem}
\theoremstyle{remark}
\newtheorem{remark}[thm]{Remark}
\theoremstyle{definition}
\newtheorem{defi}[thm]{Definition}
\newtheorem{nota}[thm]{}

\numberwithin{equation}{section}

\newtheorem{example}[thm]{Example}

\newcommand{\la}{\longrightarrow}

\newcommand{\ov}{\overline}
\newcommand{\un}{\underline}

\newcommand{\Set}{{\operatorname{Set}}^1}
\newcommand{\supp}{\operatorname{Supp}}

\newcommand{\Alb}{\operatorname{Alb}}

\newcommand{\Jac}{\operatorname{Jac}}

\def\e{\epsilon}

\newcommand{\Z}{\mathbb{Z}}
\newcommand{\R}{\mathbb{R}}
\newcommand{\N}{\mathbb{N}}

\newcommand{\SP}{\mathcal{SP}}
\newcommand{\OP}{\mathcal{OP}}

\newcommand{\Esep}{E(\Gamma)_{{\rm sep}}}
\def\md{\underline{d}}

\newcommand{\Del}{{\rm Del}}
\newcommand{\Vor}{{\rm Vor}}
\newcommand{\Fac}{{\rm Faces}}

\newcommand{\Mgt}{M_g^{\rm trop}}
\newcommand{\Agt}{A_g^{\rm trop}}
\newcommand{\Jgt}{J_g^{\rm trop}}
\newcommand{\tgt}{t_g^{\rm trop}}

\begin{document}

\title{Torelli theorem for graphs and tropical curves}

\author{Lucia Caporaso}
\address{Dipartimento di Matematica,
Universit\`a Roma Tre,
Largo S. Leonardo Murialdo 1,
00146 Roma (Italy)}
\email{caporaso@mat.uniroma3.it}

\author{Filippo Viviani}
\address{Dipartimento di Matematica,
Universit\`a Roma Tre,
Largo S. Leonardo Murialdo 1,
00146 Roma (Italy)}
\email{viviani@mat.uniroma3.it}

%\keywords{Stable curves, Compactified Picard scheme,
%Theta divisor,
%Voronoi compactification,  Stable semi-abelic pairs
% Alexeev compactification}
%\subjclass[2002]{14H40, 14H51, 14K30, 14D20}

\maketitle
\begin{center}\today

\end{center}

%\tableofcontents

\begin{abstract}
Algebraic curves have a discrete analogue in finite graphs.
Pursuing this
analogy we prove a Torelli theorem for graphs. Namely,
we show that two graphs have the same Albanese torus if and only if the
graphs obtained  from them  by contracting all  separating edges are 2-isomorphic.
In particular,   the  strong Torelli  theorem holds for
$3$-connected graphs.
Next, using the correspondence between compact tropical curves and metric  graphs,  we prove a tropical Torelli theorem   giving necessary and sufficient
conditions for two tropical curves to have the same principally polarized tropical Jacobian.
By contrast, we prove that, in a suitably defined sense, the tropical Torelli map has degree one.
Finally
we describe some natural posets associated to a graph and prove
that they characterize its  Delaunay decomposition.
\end{abstract}
\section{Introduction}

The
 analogy between graphs and algebraic curves has been a source of inspiration
 both in combinatorics and algebraic geometry.
 In this frame of mind,
 M. Kotani and T. Sunada (see \cite{KS})
introduced the  Albanese torus, $\Alb(\Gamma)$,
and the Jacobian torus, $\Jac(\Gamma)$,  of a graph $\Gamma$;
see section~\ref{at} for the precise definition.

By  \cite{KS},  $\Alb(\Gamma)$   and   $\Jac(\Gamma)$ are dual flat tori of dimension
  equal to $b_1(\Gamma)$, the first  Betti number  of
$\Gamma$. As $b_1(\Gamma)$ is the maximum number of linearly  independent cycles in $\Gamma$,
it can be viewed as the analog for a graph of the genus of a Riemann
surface.
In analogy with  the classical Torelli theorem for curves, it is natural to ask the following question:

\begin{prob}
\label{p1}
When are two graphs $\Gamma$ and $\Gamma'$  such that $\Alb(\Gamma)\cong \Alb(\Gamma')$?
\end{prob}
There exist in the literature other versions of such a problem
(see for example  \cite{BdlHN}, or   \cite{BN}); the statement of Problem~\ref{p1}
 is due to T. Sunada.
One of the goals of this paper is to   answer   the above question.
In our  Theorem~\ref{main-thm},
we prove that $\Alb(\Gamma)\cong \Alb(\Gamma')$ if and only if the two graphs obtained
from $\Gamma$ and $\Gamma'$ by contracting all of their separating edges
are cyclically equivalent (or  2-isomorphic,  cf. Definition~\ref{cyc-equiv}).

Using a result of Whitney, we obtain  that the Torelli theorem is true for $3$-connected
graphs; see Corollary \ref{main-cor}.
This answers a problem implicitly posed in \cite[Page 197]{BdlHN}, where the authors
ask, albeit indirectly, whether there exist two non isomorphic, $3$-connected   graphs
with isomorphic Albanese torus.

Let us now turn to another, recently discovered aspect of the analogy between graphs and curves,
that is, the tight connection between tropical curves and graphs.
By results of G. Mikhalkin and I. Zharkov,  see \cite{MIK3} and  \cite{MZ}, there exists  a natural bijection between the set of tropical equivalence classes of compact tropical curves  and metric graphs
all of whose vertices have valence at least 3.

Observe now that compact tropical curves, just like compact Riemann surfaces,
are endowed with a Jacobian variety, which is a principally polarized  tropical Abelian variety; see Section~\ref{trp} for details.
  The following Torelli-type question arises

\begin{prob}
\label{p2}
Can two compact tropical curves     have
isomorphic Jacobian varieties? If so, when?
\end{prob}
It is well known  (see  \cite[Sect. 6.4]{MZ})  that  the answer to the first part of this question is ``yes".
 In Theorem~\ref{main-thmt} we precisely characterize which tropical curves
  have  the same Jacobian variety.
In particular, we prove that for curves whose associated graph is 3-connected,
the   Torelli theorem holds in strong form, i.e. two such curves are tropically equivalent if and only if their polarized Jacobians are isomorphic. 

The proof of Theorem~\ref{main-thmt} is based on a Torelli theorem for metric graphs,
 Theorem~\ref{main-thml}, which is interesting in its own right, and uses essentially the same ideas as the proof of Theorem~\ref{main-thm}.
 The statement of  Theorem~\ref{main-thml} is slightly more technical, but can be phrased as follows:
 two metric graphs have the same Albanese torus if and only if
 they have the same 3-edge connected class (defined in \ref{3ec} and \ref{3ecl}).

A key ingredient   turns
out to be the Delaunay decomposition $\Del(\Gamma)$ of a graph $\Gamma$.
$\Del(\Gamma)$ is well known to be a powerful tool, and has been investigated
 in, among others, \cite{nam}, \cite{OS} and \cite{alex},  which have been quite useful in the writing of this paper.
 In Proposition \ref{Del-equ}, we characterize when two graphs have the same Delaunay decomposition.

The last section of the paper gives   other characterizations of
a graph, or rather, of the 3-edge connected class of a graph.
These  characterizations, given in  Theorem~\ref{final-thm}, use  three remarkable posets (i.e. partially ordered sets), $\SP_{\Gamma}$, $\OP_{\Gamma}$  and
$\ov{\OP_{\Gamma}}$.
The poset $\SP_{\Gamma}$
is the set of
spanning subgraphs of $\Gamma$ that are free from separating edges.
The  maximal elements of $\SP_{\Gamma}$ are the so-called C1-sets
(see Definition~\ref{C1}), which play  a crucial role in the  previous sections.
The two posets $\OP_{\Gamma}$ and $\ov{\OP_{\Gamma}}$, defined in Section~\ref{op},
are associated to totally cyclic orientations;
we conjecture a geometric interpretation for them in \ref{geo-conj}, relating to an interesting question posed in \cite{BdlHN}.

Not only is this last section
 related to the  Torelli theorems in the previous parts, but also, our interest in it
 is motivated by a different, open, Torelli problem.
 The material of Section~\ref{pos}  will in fact be applied  in our ongoing project,
\cite{CV}, in order to
describe the combinatorial structure of the compactified Jacobian of a singular  algebraic curve,
and   generalize the Torelli theorem to  stable  curves.

In the Appendix, assuming some natural facts
about the Torelli map $\tgt:\Mgt\to \Agt$
(facts that are commonly expected, yet still awaiting
to be fully settled in the literature), we prove 
that $\tgt$ is of tropical degree 
one to its image, even though it is not injective; see Theorem \ref{MZ-conj}.
This proves a conjecture of Mikhalkin-Zharkov (see \cite[Sect. 6.4]{MZ}).

\emph{Acknowledgements.}
We thank   L. Babai for a stimulating e-mail correspondence,
M. Baker for   pointing  us the paper \cite{art}, and G. Mikhalkin and I. Zharkov for
precious comments on the tropical Torelli map,   which prompted us to add the Appendix.
The second author would like to thank  T. Sunada for a series of
lectures at Humboldt University of Berlin,
during which he learnt about the Torelli problem for graphs, and
G. Mikhalkin for a series of lectures at the INdAM workshop``Geometry of
projective varieties", during which he learnt about the Torelli problem for
tropical curves.
%%r< 
Finally, we  benefitted from  a 
very thoughtful report  by
an anonymous referee, to whom we are grateful.
%%r>

\section{Preliminaries}
\subsection{The Albanese torus of a graph}
\label{at}
Throughout the paper  $\Gamma$  will be a finite  graph
(loops and multiple edges are allowed); we denote by $V(\Gamma)$ its set
of vertices    and by $E(\Gamma)$ its set of edges.

We recall the definition of the Albanese torus, from \cite{KS}.
Fix an orientation of $\Gamma$ and let $s, t: E(\Gamma)\to V(\Gamma)$ be the two maps sending an
oriented edge to its source and target point, respectively.
Notice that the Albanese torus will not depend on the chosen orientation.
Consider the spaces of chains of $\Gamma$ with values in an abelian group $A$:
$$
  C_0(\Gamma, A):=\oplus_{v\in V(\Gamma)} A \cdot v ,\hskip.8in  C_1(\Gamma, A):=\oplus_{e\in E(\Gamma)}A \cdot e .
$$
Define, as usual, a boundary map
$$\begin{aligned}
\partial : C_1(\Gamma, A) & \longrightarrow C_0(\Gamma, A) \\
e & \mapsto t(e)-s(e).
\end{aligned}$$
The first homology group
of $\Gamma$ with values in $A$ is   $H_1(\Gamma, A):=\ker\partial$.

If $A=\R$, we define the scalar product $(,)$ on $C_1(\Gamma,\R)$ by
$$(e,e')=\begin{cases}
1 & \text{ if } e=e', \\
0 & \text{ otherwise.}
\end{cases}
$$
 We continue to denote by $(,)$
 the induced
scalar product on    $H_1(\Gamma,\R)$.
The subspace $H_1(\Gamma,\Z)$
is a lattice inside $H_1(\Gamma,\R)$.

\begin{defi}\cite{KS}
\label{alb}
The {\it Albanese torus} $\Alb(\Gamma)$ of $\Gamma$ is
$$
\Alb(\Gamma):=\Bigl(H_1(\Gamma,\R)/H_1(\Gamma,\Z); (,)\Bigr)
$$
 with the flat metric derived from the scalar product $(,)$.
 \end{defi}

We have  $\dim \Alb(\Gamma)=b_1(\Gamma)$ where $b_1(\Gamma)$ is   the  first  Betti number:
$$b_1(\Gamma)={\rm{rank}_{\Z}} H_1(\Gamma, \Z)=\#\{\text{connected components of } \Gamma\} -
\#V(\Gamma) +\#E(\Gamma).
$$

There is also the cohomological version of the previous construction,
(we refer to \cite{KS} for the details).
One obtains another torus, called the {\it Jacobian torus} $\Jac(\Gamma)$,
which has the following form
$
\Jac(\Gamma):=(H^1(\Gamma,\R)/H^1(\Gamma,\Z);\langle,\rangle).
$
As we said,  $\Jac(\Gamma)$ and $\Alb(\Gamma)$ are dual flat tori.

There exist in the literature several definitions  of  Albanese and  Jacobian      torus
of a graph,
related to one another by means of standard dualities.
In particular, we need to briefly explain the relation with   \cite{BdlHN}.
Our lattice $H^1(\Gamma,\Z)$ is the dual lattice, in $(H^1(\Gamma,\R);\langle,\rangle)$,
of the so-called lattice of
integral flows $\Lambda^1(\Gamma)\subset  H^1(\Gamma,\R)$
studied in \cite{BdlHN}. In particular, the Albanese torus $\Alb(\Gamma)$ determines
the lattice $\Lambda^1(\Gamma)$ and conversely (see Proposition 3 of loc.cit.).

\subsection{Cyclic equivalence and connectivity}
\label{Sec2.2}

\begin{nota}
\label{notgraph}
We  set some notation that will be used throughout.
Let $S\subset E(\Gamma)$ be a subset of edges of a graph $\Gamma$.
We associate to   $S$ two graphs, denoted $\Gamma\smallsetminus S$
and $\Gamma(S)$, as follows

$\bullet$ The graph $\Gamma\smallsetminus S$ is, as the notation indicates,
  obtained from $\Gamma$ by removing the edges in $S$ and by leaving the   vertices unchanged.
Thus $V(\Gamma\smallsetminus S)=V(\Gamma)$
(so that  $\Gamma\smallsetminus S$ is a    spanning subgraph)
and
$E(\Gamma\smallsetminus S)=E(\Gamma)\smallsetminus S$.

%%r<
$\bullet$ The graph $\Gamma(S)$ is obtained from $\Gamma$ by contracting  all the
edges not in $S$, so that the set of edges of $\Gamma(S)$ is equal to $S$. There is a surjective contraction
map $\Gamma \to \Gamma(S)$  which contracts to a point every connected component of $\Gamma \smallsetminus S$.
Notice that   $\Gamma(S)$ is connected if and only if so is $\Gamma$.
For example, $\Gamma(E(\Gamma))=\Gamma$, and,
if $c$ is the number of connected components of $\Gamma$, then   $\Gamma (\emptyset)$ is a set of $c$ isolated points (i.e. $\Gamma (\emptyset)$ has $c$ vertices and no edges).

\begin{example}\label{disGamma(S)}
Here  is an example of a $\Gamma(S)$, with the   contraction map  $\Gamma\to\Gamma(S)$,
where $S=\{e_1, e_2\}\subset E(\Gamma)$:

\begin{figure}[!htp]
$$\xymatrix@=1pc{
&*{\bullet} \ar@{-}[rr] \ar@{-}[dd] \ar@{-}@/_/[dd] & &*{\bullet}
\ar@{-}[dd]
\ar@{-}@/_/[ll]\ar@{-}[rr]^{e_1}& & *{\bullet} \ar@{-}[dd]
\ar@{-}@/^.5pc/[rrd]\ar@{-}[drr]&&&&&&
& & &\\
\Gamma \, =  &&&&  &   &&*{\bullet} \ar@{-}@(ur,dr)
&&\ar[rr] & &&*{\bullet} \ar@{-}@/^.5pc/[rr]^{e_1} 
\ar@{-}@/_.5pc/[rr]_{e_2} &&*{\bullet} & \, = \Gamma(S) \\
& *{\bullet} \ar@{-}[rr] & &*{\bullet} \ar@{-}[rr]^{e_2}& &
*{\bullet} \ar@{-}[urr]&&
&&&&    & &&
}$$
\caption{Example of $\Gamma(S)$ with $S=\{e_1, e_2\}$.}
\end{figure}
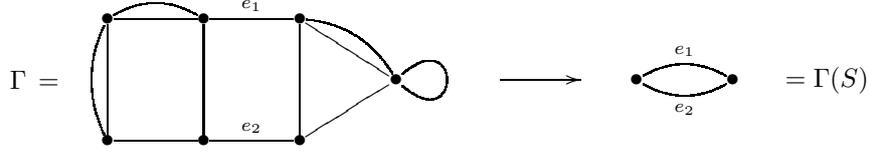
\end{example}

%%r>

We have the useful additive formula
\begin{equation}
\label{b1}
b_1(\Gamma)=b_1(\Gamma\smallsetminus S)+b_1(\Gamma(S)).
\end{equation}

If $\Gamma $ is  a connected graph, a {\it separating edge} is an   $e\in E(\Gamma)$
such that $\Gamma \smallsetminus e$ is not connected. If  $\Gamma$
is not connected
we say that an edge
is separating if it is   separating for the connected component containing it.
We denote by
$E(\Gamma)_{{\rm sep}}$ the set of  separating edges  of   $\Gamma $.
%(A separating edge is sometimes called a bridge,%an isthmus or a cut edge).

We say that a  graph $\Delta$ is a {\it cycle} if it is connected, free from separating edges and
if $b_1(\Delta)=1$. We call $\#E(\Delta)=\#V(\Delta)$ the length of $\Delta$.
\end{nota}

\begin{defi}\label{cyc-equiv}
Let  $\Gamma$ and $\Gamma'$ be two graphs.
We say that a bijection between their edges,
$\e:E(\Gamma)\to E(\Gamma ')$, is {\it cyclic} if it induces a bijection between the cycles of $\Gamma$ and the cycles of $\Gamma'$.

We say that  $\Gamma$ and $\Gamma'$ are {\it cyclically equivalent} or {\it 2-isomorphic},
 and we write $\Gamma\equiv_{\rm cyc} \Gamma'$, if there exists a cyclic bijection
  $\e:E(\Gamma)\to E(\Gamma')$.

The cyclic equivalence class of $\Gamma$ will be  denoted by $[\Gamma]_{\rm cyc}$.
\end{defi}
$[\Gamma]_{\rm cyc}$ is described by the following  result of Whitney (see also \cite[Sec. 5.3]{Oxl}).

\begin{thm}[\cite{Whi}]\label{cycequ-moves}
Two graphs $\Gamma$ and $\Gamma'$ are cyclically equivalent if and only if
they can be obtained from one  another
via   iterated applications of the following two moves:
\begin{enumerate}
\item[(1)] Vertex gluing: $v_1$ and $v_2$ are identified to the separating vertex $v$, and conversely (so that $\Gamma_1\coprod \Gamma _2 \equiv_{\rm{cyc}}\Gamma$).
\begin{figure}[!htp]
$$\xymatrix@=1pc{
&&*{\bullet} \ar@{-}[dd]\ar@{-}@/_.5pc/[dd]\ar@{-}[drr]^>{v_1}&&
&&*{\bullet}\ar@{-}[dd] \ar@{-}[dl]_>{v_2} \ar@{-}[dr] & & && &
*{\bullet} \ar@{-}[dd]\ar@{-}@/_.5pc/[dd]\ar@{-}[drr]^>{v}&&
&*{\bullet}\ar@{-}[dd] \ar@{-}[dl] \ar@{-}[dr] & &\\
\Gamma_1\ar@{=}[r]&&&&*{\bullet}&*{\bullet} \ar@{-}[rr] |!{[r]}\hole & &*{\bullet}
\ar@{=}[r]& \Gamma_2 &\equiv_{\rm{cyc}} &&
&&*{\bullet} \ar@{-}[rr] |!{[r]}\hole & &*{\bullet}\ar@{=}[r] &\Gamma          \\
&&*{\bullet} \ar@{-}[urr]&&&&*{\bullet}\ar@{-}[ru]\ar@{-}[lu]& && &
&*{\bullet} \ar@{-}[urr]&&    &
*{\bullet}\ar@{-}[ru]\ar@{-}[lu]
}$$
\caption{Two graphs $\Gamma_1$ and $\Gamma_2$ attached at $v_1\in V(\Gamma_1)$
and $v_2\in V(\Gamma_2)$.}
\label{vert-gluing}
\end{figure}
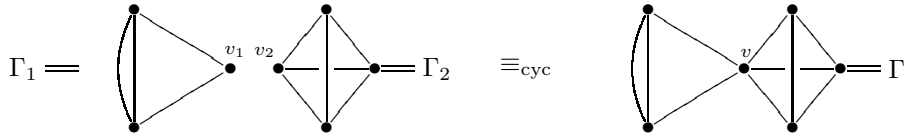
\item[(2)] Twisting:   the double arrows below mean identifications.
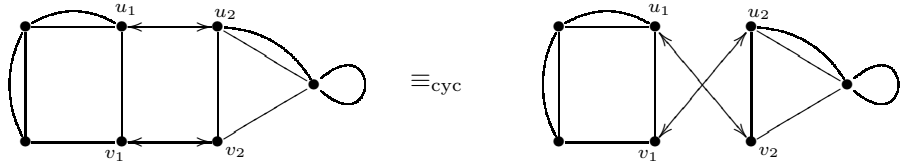
\begin{figure}[!htp]
$$\xymatrix@=1pc{
*{\bullet} \ar@{-}[rr] \ar@{-}[dd] \ar@{-}@/_/[dd] & &*{\bullet} \ar@{-}[dd]
\ar@{-}@/_/[ll]_<{u_1}\ar@{<->}[rr]& & *{\bullet} \ar@{-}[dd]
\ar@{-}@/^.5pc/[rrd]^<{u_2}\ar@{-}[drr]&&&&&&
*{\bullet} \ar@{-}[rr] \ar@{-}[dd] \ar@{-}@/_/[dd] &
&*{\bullet} \ar@{-}[dd] \ar@{-}@/_/[ll]_<{u_1} \ar@{<->}[ddrr]&&
*{\bullet} \ar@{-}[dd]\ar@{-}@/^.5pc/[rrd]^<{u_2}\ar@{-}[drr]&& \\
&&&  &   &&*{\bullet} \ar@{-}@(ur,dr)
&&\equiv_{\rm{cyc}} &&&&& &    &&*{\bullet} \ar@{-}@(ur,dr)                 \\
*{\bullet} \ar@{-}[rr]_>{v_1} & &*{\bullet} \ar@{<->}[rr]& &
*{\bullet} \ar@{-}[urr]_<{v_2}&&
&&&& *{\bullet} \ar@{-}[rr]_>{v_1} &  &*{\bullet} \ar@{<->}[uurr]&&
*{\bullet} \ar@{-}[urr]_<{v_2}&&
}$$
\caption{A twisting at a separating pair of vertices.}
\label{twist}
\end{figure}
\end{enumerate}
\end{thm}
%%r<
Let us describe the above  twisting  move   more precisely. Let $u, v$ be a pair of separating vertices of $\Gamma$. Then $\Gamma$ is obtained from two graphs, $\Gamma_1$ and $\Gamma_2$, 
by identifying two pairs of vertices as follows: Let $u_i, v_i\in V(\Gamma_i)$ for $i=1,2$.
Then $\Gamma$ is given by attaching $\Gamma_1$ to $\Gamma _2$ by the two identifications $u_1=u_2=u$
and $v_1=v_2=u$. The twisting at the pair $u, v$  is the graph $\Gamma '$ obtained
by attaching $\Gamma_1$ to $\Gamma _2$ by the two identifications $u_1=v_2$
and $v_1=u_2$. 
%%r>

\

We now recall the definitions of connectivity  (see for example \cite[Chap. 3]{Die}).
Let $k\geq 1$ be an integer. A graph $\Gamma$ having at least $k+1$-vertices is said to be $k$-{\it connected}  if the graph obtained from $\Gamma$ by removing any $k-1$ vertices, 
%%%
and all the edges adjacent to them, 
%%%
 is connected.
A graph $\Gamma$ having at least $2$-vertices is said to be $k$-{\it edge connected}  if the graph obtained from $\Gamma$ by removing any $k-1$   edges
is connected.

If $\Gamma $ is $k$-connected it is also $k$-edge connected,
but the converse fails.

$\Gamma$ is $1$-connected, or
 $1$-edge connected, if and only if it is connected.

 $\Gamma$ is $2$-edge connected if and only if it is connected and
$E(\Gamma)_{\rm sep}=\emptyset$.

3-edge connected graphs will play an important role, and will be   characterized in Corollary~\ref{cor3}.

%%r<
\begin{remark}
\label{ } We shall frequently consider edge-contracting maps, for which 
we make the following useful  observation.
Let $\Gamma\to \Gamma '$ be a (surjective) map contracting some edge of $\Gamma$ to a point. Then
if $\Gamma$ is $k$-edge connected so is $\Gamma'$.
\end{remark}
%%r>
\begin{remark}
\label{3v} If  $\Gamma$ is 3-connected,
the cyclic equivalence class of   $\Gamma$ contains only $\Gamma$.
Indeed, by Theorem~\ref{cycequ-moves} a move of type (1) can be performed only in the presence of a disconnecting vertex, and a move of type (2) in the presence of a separating pair of vertices.
\end{remark}

\subsection{C1-sets and connectivizations.}
\begin{defi}\label{C1}
Let $\Gamma$ be a graph and $S\subset E(\Gamma)$.

Suppose $\Gamma$  connected  and $\Esep=\emptyset$; we
 say that    $S$ is a \emph{C1-set} of $\Gamma$ if $\Gamma(S)$ is a cycle and if $\Gamma\smallsetminus S$ has no separating edge.

In general, let $\widetilde{\Gamma}:=\Gamma \smallsetminus E(\Gamma)_{{\rm sep}}$. We say that
$S$ is a C1-set of $\Gamma$ if $S$ is a C1-set of a connected component of $\widetilde{\Gamma}$.

We denote by $\Set \Gamma$ the set of C1-sets of $\Gamma$.
\end{defi}

For instance, the set $S$ in Example~\ref{disGamma(S)} is a C1-set.

The terminology ``C1" stands for ``Codimension 1", and   will be justified in \ref{codim}.
The following Lemma summarizes some useful properties of C1-sets.

 \begin{lemma}\label{C1lm}
Let $\Gamma$ be a graph and $e, e'\in E(\Gamma)$. Then
\begin{enumerate}[(i)]
\item
\label{C11} Every C1-set $S$ of $\Gamma$ satisfies $S\cap E(\Gamma)_{\rm sep}=\emptyset$.
\item
\label{C12} Every non-separating edge $e$ of $\Gamma$ is contained in a unique C1-set,  $S_e$.
If $\Esep=\emptyset$, then $S_e=E(\Gamma\smallsetminus e)_{\rm sep}\cup\{e\}$.
\item
\label{C13}
$e$ and $e'$ belong to the same C1-set if and only if they belong to the same cycles of $\Gamma$.
\item
\label{C14}
Assume  $\Gamma$ connected and $e$ and $ e'$  non-separating.
Then $e$ and $e'$ belong to the same C1-set  if and only if   $\Gamma\smallsetminus \{e, e'\}$ is disconnected ($(e,e')$ is called a {\emph {separating pair of edges}}).
\end{enumerate}\end{lemma}
\begin{proof}
The first assertion follows trivially from   Definition~\ref{C1}.

Notice that a C1-set of $\Gamma$ is entirely contained in the set of edges of a unique connected component of
$\widetilde{\Gamma}$.
Therefore we can assume that $\Gamma $ is connected, and, for parts (\ref{C12}) and (\ref{C14}),     free from separating edges.

Fix an edge $e\in E(\Gamma)$, let
 $\Gamma_e=\Gamma \smallsetminus e$ and
set
\begin{equation}
\label{Se}
S_e:=E(\Gamma_e)_{\rm sep}\cup\{e\}\subset E(\Gamma).
\end{equation}

We claim that $S_e$ is  the unique C1-set containing $e$.
 We have that  $\Gamma(S_e)$ is connected and free from separating edges  (as $\Gamma$ is). Therefore
to prove that $S_e$ is a C1-set  it suffices to prove that $b_1(\Gamma(S_e))=1$ .
 Let $\Gamma'$   be the graph obtained   from $\Gamma(S_e)$ by
removing $e$;  then $b_1(\Gamma ')=0$ (by construction all its edges are separating).
Now,
$
\#E(\Gamma(S_e))=\#E(\Gamma ')+1
$,
and, of course, $\Gamma(S_e)$ and $\Gamma '$ have the same  vertices.
Therefore
$
b_1(\Gamma(S_e)) = b_1(\Gamma ')+1=1
$. So $S$ is a C1-set.

Finally, let $\widetilde{S}$ be a C1-set containing $e$.
It is clear that $S_e\subset \widetilde{S}$  (any  $e'\in S_e$ such that $e'\not\in  \widetilde{S}$ would be a separating edge
of $\Gamma \smallsetminus \widetilde{S}$).
To prove that $S_e= \widetilde{S}$ consider the map
$\Gamma \to \Gamma(\widetilde{S})$ contracting all the edges not in $\widetilde{S}$.
Suppose, by contradiction, that there is an edge $\widetilde{e}\in \widetilde{S}\smallsetminus S_e$;
since $\Gamma(\widetilde{S})$ is a cycle, $\widetilde{e}$ is a separating edge of
$\Gamma(\widetilde{S})\smallsetminus e$. Therefore  $\widetilde{e}$ is a separating edge of
$\Gamma \smallsetminus e=\Gamma_e$, and hence $\widetilde{e}$ must lie in $S_e$,
by    \ref{Se}. This is a contradiction,   (\ref{C12}) is proved.

Now part (\ref{C13}). We can assume that $e$ and $e'$ are non-separating, otherwise it is obvious.
Suppose $S_{e}=S_{e'}$; then, by definition,  we can assume that $\Esep=\emptyset$.
 Let $\Delta\subset \Gamma$ be a cycle containing $e'$. By part (\ref{C12}) we have that $e'$ is a separating edge of $\Gamma \smallsetminus e$; therefore if $\Delta$ does not contain $e$,
then $e'$ is a separating edge of $\Delta$, which is impossible.
Conversely, if $e'\not\in S_{e}$ then (as $e'$ is non-separating for
$\Gamma \smallsetminus S_{e}$) there exists a cycle $\Delta \subset \Gamma \smallsetminus S_{e}$ containing $e'$. So $e$ and $e'$ do not lie in the same cycles.

Finally part (\ref{C14}).
If $(e,e')$ is a separating pair then
  $e$ is a separating edge of $\Gamma \smallsetminus  e'$ and $e'$ is a separating edge of $\Gamma \smallsetminus  e$. By part (\ref{C12}) $e$ and $e'$ belong to the same C1-set.
  The converse follows from the fact that a   cycle with two edges removed is disconnected.
\end{proof}
\begin{remark}
\label{decdelta}
Let $\Delta\subset \Gamma$ be a cycle. By Lemma~\ref{C1lm} the set $E(\Delta)$
is a  disjoint  union of C1-sets. We define
$ \Set_{\Delta}\Gamma :=\{S\in \Set \Gamma: \  S\subset E(\Delta)\}
$
so that
$$
E(\Delta)=\coprod_{S\in \Set_{\Delta}\Gamma}S.
$$
%We call $\Set_{\Delta}\Gamma$ the C1-decomposition of $\Delta$.
\end{remark}

\begin{cor}
\label{cor3}
 A graph $\Gamma$ is 3-edge connected  if and only if it is connected and there is a bijection
 $E(\Gamma)   \to  \Set  \Gamma$ mapping $e\in E(\Gamma)$ to $\{ e\}\in\Set  \Gamma$.
\end{cor}
\begin{proof}
If $\Gamma$ is  3-edge connected it is free from separating edges; hence every $e\in E(\Gamma)$ belongs to a unique $S\in \Set\Gamma$.
So it suffices to prove that every $S\in \Set \Gamma$ has cardinality 1.
Suppose there are two distinct edges $e,e'\in S$.   Then Lemma~\ref{C1lm}(\ref{C14})
  yields that $\Gamma  \smallsetminus \{e,e'\}$ is not connected, which is a contradiction.

Conversely, if every edge lies in a C1-set, then $\Gamma$ has no separating edges.
If $\Gamma$ is not 3-edge connected, it admits a separating pair of edges $(e,e')$. Then $e$ and $e' $ belong to the same $S\in \Set \Gamma$
(by Lemma~\ref{C1lm}). So we are done.
\end{proof}
In the next statement we use the notation of  \ref{C1lm}(\ref{C12}) and  \ref{cyc-equiv}.
\begin{cor}
\label{CC1}
Let $\Gamma $ and $\Gamma '$ be cyclically equivalent; then $\#\Esep = \#E(\Gamma')_{\rm{sep}}$.
Let
 $\e:E(\Gamma)\to E(\Gamma')$ be a cyclic bijection; then
  $\e$ induces a bijection
  $$
 \begin{aligned}
\beta_{\e}: &\  \Set \Gamma  & \la &\Set  \Gamma ' \\
&\  S_e &\longmapsto & \  S_{\e (e)}\
 \end{aligned}
 $$
 such that $\#S =\#\beta_{\e}(S)$ for every $S\in \Set \Gamma$.
\end{cor}
\begin{proof}
An edge is separating if and only if it is not contained
in any cycle. Therefore $\epsilon$ maps $\Esep $ bijectively
to $E(\Gamma')_{\rm{sep}}$, so the first part is proved.
The second part follows immediately from Lemma~\ref{C1lm} (\ref{C12})
and (\ref{C13}).
\end{proof}

We introduce two types of edge contractions  that will be used extensively later:

%\pagebreak

\begin{enumerate}[(A)]
\item
\label{A} Contraction of a separating edge:

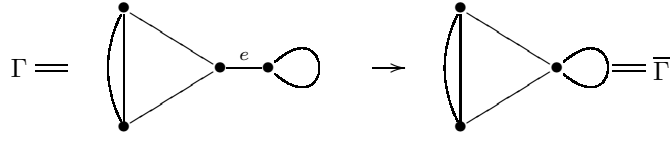
\begin{figure}[!htp]
$$\xymatrix@=1pc{
&&*{\bullet} \ar@{-}[dd]\ar@{-}@/_.5pc/[dd]\ar@{-}[drr]&&&&&&&
*{\bullet} \ar@{-}[dd]\ar@{-}@/_.5pc/[dd]\ar@{-}[drr]&&&&\\
\Gamma\ar@{=}[r]&&&&*{\bullet}\ar@{-}[r]^{e}&*{\bullet} \ar@{-}@(ur,dr)&&\ar@{->}[r]&&
&&*{\bullet}\ar@{-}@(ur,dr)& \ar@{=}[r]& \ov{\Gamma}\\
&&*{\bullet} \ar@{-}[urr]&&&&&&& *{\bullet} \ar@{-}[urr]&&&&
}$$
\caption{The contraction of the separating edge $e\in E(\Gamma)$.}\label{cont-sep}
\end{figure}

\item
\label{B} Contraction of one of two edges of a separating pair of edges:

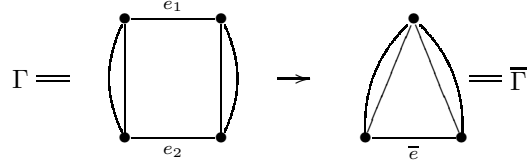
\begin{figure}[!htp]
$$\xymatrix@=1.pc{
&&*{\bullet} \ar@{-}[rr]^{e_1} \ar@{-}[dd] \ar@{-}@/_/[dd] & &*{\bullet}
\ar@{-}[dd] \ar@{-}@/^/[dd]
& &&& *{\bullet} \ar@{-}[ddl] \ar@{-}@/_/[ddl] \ar@{-}@/^/[ddr] \ar@{-}[ddr]&& \\
\Gamma\ar@{=}[r]&&&&&\ar@{->}[r]&& &&\ar@{=}[r]&\ov{\Gamma}\\
&&*{\bullet} \ar@{-}[rr]_{e_2} & &*{\bullet}& && *{\bullet} \ar@{-}[rr]_{\ov{e}}
& & *{\bullet}&\\
}
$$
\caption{The contraction of the edge $e_1$ of the separating pair
$(e_1, e_2)$.}\label{con-pair}
\end{figure}
\end{enumerate}

To a graph $\Gamma$ we shall associate two types of graphs.

\begin{defi}\label{3-conn}
The {\it $2$-edge connectivization} of a connected graph $\Gamma$
is the $2$-edge connected graph $\Gamma^2$ obtained from $\Gamma$ by iterating
the above operation (A) (for all the separating edges of $ \Gamma$).

A {\it $3$-edge connectivization} of a connected graph $\Gamma$ is a
  $3$-edge connected graph $\Gamma^3$ which
 is obtained from $\Gamma^2$ by iterating
the above operation (B).

If $\Gamma$ is not connected, we define $\Gamma^2$
(resp.   $\Gamma^3$)  as the disjoint union of the $2$-edge connectivizations
(resp. $3$-edge connectivizations)
of its connected components.
\end{defi}
\begin{remark}
It is clear that $\Gamma^2$ is uniquely determined, while $\Gamma^3$ is not.

If $\Gamma$ is not connected  $\Gamma^2$
(resp. $\Gamma^3$) is   not $2$-edge  (resp. $3$-edge) connected.

There is a (surjective) {\it contraction map} $\sigma:\Gamma \to \Gamma^2
\to \Gamma^3$
obtained by composing the contractions defining $\Gamma^2$ and $\Gamma^3$.
\end{remark}
\begin{lemma}
\label{3lm}
Let $\Gamma$ be a  graph.
\begin{enumerate}[(i)]
%\item\label{3lme}$\Gamma$ admits a  3-edge connectivization $\Gamma^3$.
\item
\label{3lmb}
 $b_1(\Gamma^3)=b_1(\Gamma^2)=b_1(\Gamma)$.
 %C for every 3-edge connectivization $\Gamma^3$.
\item
\label{3lmC1}
There are canonical bijections
$$
\Set \Gamma^3 \leftrightarrow E(\Gamma^3)\leftrightarrow \Set \Gamma.
$$
\item
\label{3lmcyc}
Two 3-edge connectivizations of $\Gamma$ are cyclically equivalent.
\item
$\Gamma^2\equiv_{\rm cyc}\Gamma\smallsetminus E(\Gamma)_{\rm sep}.$
\end{enumerate}
\end{lemma}
\begin{proof}
The first Betti number is invariant under the operations (\ref{A}) and (\ref{B}) above, because  no loop  gets contracted.
So, part  (\ref{3lmb}) is done.
Notice also (which will be used later) that the contraction map
$\sigma:\Gamma \to \Gamma^3$ induces a natural bijection between the cycles of $\Gamma$ and those of $\Gamma^3$.

Now part  (\ref{3lmC1}). The bijection $\Set \Gamma^3 \leftrightarrow E(\Gamma^3)$ is described in \ref{cor3}.
Let $S\in \Set \Gamma$ and set
\begin{equation}
\label{Sset}
S=\{ e_{S,1},\ldots, e_{S,\# S} \}.
\end{equation}
Consider again the contraction map
$\sigma:\Gamma \to \Gamma^3$. Clearly $\sigma$ contracts all the edges of $S$ but one,
which gets mapped to an edge $e _S\in E(\Gamma^3)$. We have thus defined a map
\begin{equation}
\label{Smap}
\psi:\Set \Gamma \la  E(\Gamma^3);\  \  \  S\longmapsto e _S.
\end{equation}
By \ref{C1lm} and by the definition of $\Gamma^3$ the above map is a bijection. So (\ref{3lmC1}) is proved.

Let $\Gamma^3$ and $\widetilde{\Gamma^3}$ be two  3-edge connectivizations
of $\Gamma$. By part  (\ref{3lmC1}) there is a natural bijection
$E(\Gamma^3)\leftrightarrow E(\widetilde{\Gamma^3})$.
Moreover, by what we said before, the two contraction maps
$$
\sigma: \Gamma \la {\Gamma^3} \hskip.7in   \widetilde{\sigma}: \Gamma \la \widetilde{\Gamma^3}\
$$
induce natural bijections between cycles that are compatible with the bijection $E(\Gamma^3)\leftrightarrow E(\widetilde{\Gamma^3})$. Therefore $\Gamma^3 \equiv_{\rm cyc} \widetilde{\Gamma^3}$, and
the part (\ref{3lmcyc}) is proved.

For the last part, it suffices to observe that $\Gamma^2$ can be obtained by $\Gamma\smallsetminus E(\Gamma)_{\rm sep}$
by    moves of type (1) (vertex gluing) in Theorem \ref{cycequ-moves}.
\end{proof}

\begin{prop}\label{lift-cyc} Let $\Gamma$ and $\Gamma '$ be two graphs.
\begin{enumerate}[(i)]
 \item \label{lift1}
Assume   $ \Gamma^2 \equiv_{\rm cyc} \Gamma'^2$. Then  $ \Gamma  \equiv_{\rm cyc}  \Gamma' $ if and only if $ \# E(\Gamma)_{\rm sep}= \# E(\Gamma')_{\rm sep}$.

\item \label{lift2}
Assume $ \Gamma^3 \equiv_{\rm cyc} \Gamma'^3$
 and  $\Esep = E(\Gamma ')_{\rm{sep}}=\emptyset$.
Then  $ \Gamma  \equiv_{\rm cyc}  \Gamma' $ if and only if
the natural bijection
$$
\beta: \Set \Gamma \stackrel{\psi}{\la} E( \Gamma^3) \stackrel{\e^3}{\la}
E( \Gamma'^3)
 \stackrel{(\psi')^{-1}}{\la}  \Set \Gamma'
$$
satisfies $\#S=\#\beta (S)$,
where $\psi$ and $\psi' $ are the bijections defined in (\ref{Smap}), and $\e^3$ a cyclic bijection.
\end{enumerate}
\end{prop}
\begin{proof}
The ``only if" part for both (\ref{lift1}) and (\ref{lift2}) holds in general, by Corollary~\ref{CC1}.
It suffices to add (for part (\ref{lift2})) that any cyclic bijection $\e:E(\Gamma)\to E(\Gamma ')$ induces a canonical cyclic bijection $\e^3:E( \Gamma^3) \to E( \Gamma'^3)$, and it is clear that
$(\psi')^{-1}\circ \e^3 \circ \psi =\beta_{\e}$ defined in \ref{CC1}.

Let us prove the sufficiency for part (\ref{lift1}). The point is that we can identify the edges
of $\Gamma^2$ with   the non-separating edges of $\Gamma$ so that we have
$E(\Gamma)=E( \Gamma^2)\coprod \Esep$; the same holds for $\Gamma '$ of  course. So,
pick a cyclic bijection $\e^2:E( \Gamma^2) \to E( \Gamma'^2)$ and
any bijection  $\e_{\rm{sep}}:E( \Gamma )_{\rm{sep}} \to E( \Gamma' )_{\rm{sep}}$.
Then we can glue $\e^2$ with $\e_{\rm{sep}}$ to a
bijection  $\e:E(\Gamma)\to E(\Gamma ')$ which is easily seen to be cyclic.

Now we prove the sufficiency in part (\ref{lift2}).
Recall that the contraction map $\sigma:\Gamma \to \Gamma ^3$ induces a natural bijection
between the cycles of $\Gamma$ and the cycles of $\Gamma ^3$; and the same holds for $\Gamma '$.
Therefore $\e^3$ induces a bijection, call it $\eta$, between the cycles of
$\Gamma$ and the cycles of $\Gamma '$.

On the other hand the bijection $\beta$ in the statement
induces a (non unique) bijection between $\e:E(\Gamma)\to E(\Gamma')$. Indeed, as
$\Gamma$ and $\Gamma' $ have no separating edges, every edge belongs to a unique C1-set (\ref{C1lm}). As $\beta$ preserves the cardinality of the C1-sets,
we easily obtain our $\e$. To show that $\e$ is cyclic, it suffices to observe that,
because of the naturality of the various maps,
$\e$ induces the above bijection $\eta$ between cycles of $\Gamma$ and $\Gamma '$.
\end{proof}
\begin{remark}
\label{3ec}
By the previous results, the class $[\Gamma^3]_{\rm cyc}$ depends solely on
$[\Gamma]_{\rm cyc}$.

Moreover, every representative in the class $ [\Gamma^3]_{\rm cyc}$
is such that each of its connected components is 3-edge-connected; therefore we shall refer to $[\Gamma^3]_{\rm cyc}$
as the {\it 3-edge connected class of $\Gamma$}.
\end{remark}

 \subsection{Totally cyclic orientations}
\begin{defi}
\label{tot}
Let $\Gamma$ be a graph and   $V(\Gamma)$ its set of vertices.

 If $\Gamma$ is connected, we
say that an orientation of $\Gamma$ is {\it totally cyclic} if   there exists
no proper non-empty subset $W\subset V(\Gamma)$ such that the edges between $W$ and
its complement  $V(\Gamma)\smallsetminus W$ go all in the same direction.
%%r<
i.e. either all from $W$ to $V(\Gamma)\smallsetminus W$, or all in the opposite direction.
%%r>

 If $\Gamma$ is not connected, we say that an orientation of $\Gamma$
is totally cyclic if   the orientation induced on each connected component
of $\Gamma$ is totally cyclic.
\end{defi}
Other names
 for
these orientations are ``strongly connected", and ``stable" (the latter is used in algebraic geometry).
\begin{remark}
\label{orex}
A cycle $\Delta$ admits exactly two totally cyclic orientations, which are usually called
just  cyclic, for obvious reasons.

On the other hand   if $\Esep\neq \emptyset$ then $\Gamma$ admits no totally cyclic orientations.
Indeed, suppose $\Gamma$ connected for simplicity and let $e\in \Esep$. Then the graph $\Gamma\smallsetminus e$ is the disjoint union of two graphs $\Gamma_1$ and $\Gamma_2$.
Then the set $W=V(\Gamma_1)\subset V(\Gamma)$ does not satisfy the requirement of Definition~\ref{tot}.
\end{remark}
The following lemma, the first part of which is already known, will be very useful.

\begin{lemma}\label{chso}
Let $\Gamma$ be a  graph.
\begin{enumerate}[(1)]
\item
 \label{c2}$\Gamma$ admits a totally cyclic orientation if and only if $E(\Gamma)_{\rm sep}=\emptyset$.
 \item
 \label{c1}
Assume $E(\Gamma)_{\rm sep}=\emptyset$  and fix
   an orientation  on $\Gamma$. The following conditions
are equivalent:
\begin{enumerate}[(a)]
\item
 \label{totcyc}
The orientation   is totally cyclic.
\item
 \label{vw}
 For any distinct $v, w\in V(\Gamma)$ belonging to the same connected component of
$\Gamma$, there exists a  path   oriented from $w$ to $v$.
\item
 \label{bto}
$H_1(\Gamma,\Z)$ has a basis of cyclically oriented cycles.
\item
\label{cc}
Every edge $e\in E(\Gamma)$ is contained in a cyclically oriented cycle.

\end{enumerate}
\end{enumerate}
\end{lemma}
\begin{proof}
Part (\ref{c2}). We already observed, in \ref{orex}, that if $\Gamma$ has a separating edge    it
does not admit
a totally cyclic orientation.
The converse, which is the nontrivial part, was proved in
\cite{robbins}, or later in
 \cite[Lemma 1.3.5]{ctheta}.

We now prove the equivalence of the four conditions in Part
(\ref{c1}).

(\ref{totcyc})$\Rightarrow $ (\ref{vw})
Pick $w\in V(\Gamma)$. Let $W\subset V(\Gamma)$ be the set of all  vertices $v$ such that $\Gamma$ contains a path oriented from $w$ to $v$.
We want to prove that $W=V(\Gamma)$. By contadiction, suppose that $V(\Gamma)\smallsetminus W$ is not empty.
Then every edge $e$ joining a vertex $w'$ in $W$ with a vertex $v$ in $V(\Gamma)\smallsetminus W$
must be oriented from $v$ to $w'$ (otherwise the path obtained attaching $e$ to an oriented  path from $w$ to $w'$ would be oriented from $w$ to $v$, so that $v\in W$, which is not the case). But then 
every edge between $V(\Gamma)\smallsetminus W$ and $W$ goes from the former to the latter, hence
the orientation is not totally cyclic, a contradiction.

(\ref{vw})$\Rightarrow $ (\ref{bto}).
Let $d\in H_1(\Gamma, \Z)$ be an element  corresponding to a cycle $\Delta$.
We claim that $d$ can be expressed as $d=\sum n_id_i$ with each $d_i$ corresponding to a cyclically oriented cycle 
$\Delta_i\subset \Gamma$, and $n_i \in \Z$.
Suppose that $\Delta $ is not cyclically oriented (otherwise there is nothing to prove). 
Clearly every edge of $\Delta$ is contained in a unique maximal oriented (connected) path
contained in $\Delta$. This enables us to
 express
$\Delta$ as a union of maximal oriented paths, call them $p_1,\ldots, p_c$,  such that
every $p_i$ is adjacent to $p_{i-1}$ and $p_{i+1}$ (with the cyclic convention $p_0=p_c$; note that $c\geq 2$).
More precisely, call $v_1,\ldots , v_c$ the vertices of this decomposition,
so that $v_i, v_{i+1}$ are the end points of $p_i$ for every $i<c$ and $v_c, v_1$ are the end points of $p_c$.
Call $s_i$, respectively  $t_i$, the starting, respectively the ending, vertex of each path.
With no loss of generality, we may assume that $s_1=v_1$ (i.e.    $p_1$ starts from $v_1$)
and that $d = \sum _{i=1}^c(-1)^{i-1}p_i$ (abusing notation slightly).
By the  maximality assumption, we  obtain  that every odd vertex is the source of both  its adjacent paths, and every even vertex is the target of its adjacent paths, i.e.:
$$
v_{2i+1}=s_{2i}=s_{2i+1},\  \  \  v_{2i}=t_{2i}=t_{2i-1}.
$$
Notice that the number of paths, $c$, is necessarily even.

Now, by (\ref{vw}) we can pick a set of paths, $q_1,\ldots, q_{c-1}$, in   $\Gamma$ such that $q_i$ joins $v_1$ and  $v_{i+1}$
and is oriented as follows. 
For every odd $i$ the  path $q_i$ starts from $v_{i+1}$ and ends in $v_1$.
 For every even $i$ the  path $q_i$ starts from $v_1$ and ends in  $v_{i+1}$.

With this choice, we have the following cyclically oriented cycles $\Delta_1,\ldots \Delta_c$.
The cycle $\Delta_1$ 
is obtained by composing the paths $p_1$ and $q_1$; for all $1<i<c$ the cycle $\Delta_i$
is obtained by composing the paths $p_i$, $q_i$ and $q_{i-1}$; finally $\Delta_c$ is the composition of $p_c$ with $q_{c-1}$.
We have
$$
d=\sum _{i=1}^c(-1)^{i-1}p_i= p_1+q_1-\sum_{i=2}^{c-1}(-1)^i(p_i+q_i+q_{i-1})-p_c-q_{c-1}=
\sum_{i=1}^c(-1)^{i-1}d_i
$$
where $d_i\in H_1(\Gamma, \Z)$ corresponds to $\Delta_i$.
This proves that the $\Z$-span of the set of cyclically oriented cycles is the entire $H_1(\Gamma, \Z)$.

(\ref{bto})$\Rightarrow $ (\ref{cc}).
Pick a basis   of cyclically oriented cycles for $H_1(\Gamma, \Z)$.
By contradiction, let $e\in E(\Gamma)$ be such that there exists no cyclically oriented cycle containing it.
Then  there exists no basis element containing $e$, and hence $e$ is not contained in any cycle, which is obviously impossible
(as $E(\Gamma)_{\rm sep}=\emptyset$).

(\ref{cc})$\Rightarrow $ (\ref{totcyc}).
By contradiction, assume there exists a set of vertices $W$ such that $\emptyset\subsetneq W \subsetneq V(\Gamma)$
and such that every edge between $W$ and $V(\Gamma)\smallsetminus W$ goes from $W$ to $V(\Gamma)\smallsetminus W$.
Let $e$ be any such edge; every cycle $\Delta$
containing $e$ must contain another edge $e'$ between  $W$ and $V(\Gamma)\smallsetminus W$, and therefore (as $e'$ is also oriented from
$W$ to $V(\Gamma)\smallsetminus W$) $\Delta$ is not cyclically oriented. We conclude that no cycle containing $e$ is cyclically oriented, and this contradicts part (\ref{cc}).
\end{proof}

We shall use the following notation.
 For any edge $e\in E(\Gamma)$, we denote by $e^*\in C_1(\Gamma, \R)^*$ the functional
on $C_1(\Gamma, \R)$ defined, for $e'\in E(\Gamma)$,
\begin{equation}
\label{e*}
e^*(e')=\begin{cases}
1 & \text{ if } e'=e, \\
0 & \text{ otherwise.}
\end{cases}
\end{equation}
We shall constantly abuse notation  by  calling
$e^*\in H_1(\Gamma, \R)^*$ also  the restriction of $e^*$ to $ H_1(\Gamma, \R)$.
\begin{remark}
\label{e0sep}
$e\in \Esep$ if and only if the restriction of $e^*$ to $H_1(\Gamma, \R)$ is zero.

Indeed  $e\in \Esep$ if and only if $e$ is not contained in any cycle of $\Gamma$.
 \end{remark}

Recall that for any $S\in \Set \Gamma$ we denote
$S=\{ e_{S,1},\ldots, e_{S,\# S} \}.$

\begin{cor}
\label{cH}
Let $\Gamma$ be a graph and fix an orientation   inducing a totally cyclic orientation on
$\Gamma\smallsetminus E(\Gamma)_{\rm sep}$. Then the following facts hold.
\begin{enumerate}
\item
\label{cH1}For every $c\in H_1(\Gamma, \Z)$
we have
$$
c=\sum_{S\in \Set \Gamma} r_S(c)\sum _{i=1}^{\#S} e_{S,i},\  \  \  \  r_S(c)\in \Z.
$$
\item
\label{cH2}
Let   $e_1, e_2\in E(\Gamma)\smallsetminus \Esep$.
There exists $u\in \R$ such that $e_1^*=u e_2^*$ on $H_1(\Gamma, \R)$ if and only if $e_1$ and $e_2$ belong to the same
C1-set of $\Gamma$; 
%%r<
moreover, in this case 
%%r>
 $u=1$.
\end{enumerate}
 \end{cor}

\begin{proof}
Let $\Delta\subset \Gamma$ be a cyclically oriented cycle. Then
$\sum_{e\in E(\Delta)}e\in H_1(\Gamma, \Z)$.
By  Lemma~\ref{C1lm}
(\ref{C13}),  if a C1-set intersects the set of edges of a cycle, then it is entirely contained in it.
So part (\ref{cH1}) follows from
 Lemma~\ref{chso} (\ref{bto}).

 For the second part,   if  $e_1$ and $e_2$ belong to the same
C1-set then  $e_1^*=e_2^*$  by the first part.
Conversely, suppose $e_1$ and $e_2$ belong to different C1-sets, $S_1$ and $S_2$. Then by Lemma~\ref{C1lm} (\ref{C13}) there exists a cycle containing $e_1$ and not $e_2$.
Hence there exists $c\in H_1(\Gamma, \Z)$ such that
$r_{S_1}(c)\neq 0$ and $r_{S_2}(c) =0$. But then $e_1^*(c)=r_{S_1}(c)\neq 0$ and
$e_2^*(c)=r_{S_2}(c)= 0$ therefore $e_1^*\neq u e_2^*$ for any $u\in \R$.
 \end{proof}

\section{Torelli theorem for graphs }
\subsection{Statement of the   theorem}
The aim of this section is to prove the following Torelli theorem for graphs.

\begin{thm}\label{main-thm}
Let $\Gamma$ and $\Gamma'$ be two
%C connected
graphs.
Then  $\Alb(\Gamma)\cong \Alb(\Gamma')$
if and only if $ \Gamma^2\equiv_{\rm cyc} \Gamma'^2$.
\end{thm}

We deduce that the Torelli theorem is true in stronger form for $3$-connected graphs.
More generally:

\begin{cor}\label{main-cor}
Let $\Gamma$ be   $3$-connected
   and let $\Gamma'$ have no vertex of valence 1.
Then $\Alb(\Gamma)\cong \Alb(\Gamma')$
if and only if $ \Gamma \cong \Gamma'$.
\end{cor}
\begin{proof}
By hypothesis $\Gamma^2= \Gamma$.
Assume $\Alb(\Gamma)\cong \Alb(\Gamma')$;
then Theorem \ref{main-thm} yields $\Gamma\equiv_{\rm cyc}   \Gamma'^2$.
By Remark~\ref{3v} we obtain $\Gamma\cong\Gamma'^2$. If $\Gamma'^2\not\cong \Gamma'$,
then the contraction map $\Gamma '\to \Gamma'^2$ certainly produces some separating vertex,
 given by the image of a separating edge of $\Gamma'$
 (because $\Gamma '$ has non vertex of valence 1). But $\Gamma'^2$ has no such vertices, by the assumption on $\Gamma$.
Hence we necessarily have $\Gamma'\cong\Gamma'^2\cong\Gamma$.
\end{proof}

\begin{proof}[Proof of Theorem \ref{main-thm}: sufficiency.]
The ``if"  direction of Theorem~\ref{main-thm} is not difficult, and it follows from the subsequent statement, part (\ref{Alb1}) of which  is already known; see \cite[Prop. 5]{BdlHN}
(where a different language is used).
\begin{prop}\label{Alb-eq}
Let $\Gamma$ be a
%C connected
graph.
\begin{enumerate}[(i)]
 \item \label{Alb1}
$\Alb(\Gamma)$ depends only on $[\Gamma]_{\rm cyc}$.
\item \label{Alb2}
$\Alb(\Gamma)=\Alb(\Gamma^2)$.
\end{enumerate}
\end{prop}
\begin{proof}
Part (\ref{Alb1}) follows from the fact that
$(H_1(\Gamma,\Z);(,))$ is defined entirely in terms
of the inclusion $H_1(\Gamma, \Z)\subset C_1(\Gamma, \Z)$ and of the basis $E(\Gamma)$
of $C_1(\Gamma,\Z)$, which is clearly invariant by cyclic equivalence.

For the second part,
first note that we can naturally identify
\begin{equation}
\label{inc}
E(\Gamma^2)=E(\Gamma)\smallsetminus E(\Gamma)_{\rm sep}\subset E(\Gamma).
\end{equation}
We fix
orientations on $\Gamma^2$ and    $\Gamma$ that are compatible   with respect to
the above (\ref{inc}).
It is clear that there is a   natural commutative diagram
\begin{equation}\label{diag1}
\xymatrix{
H_1(\Gamma^2,\Z) \ar@{^{}->}_{\cong}^{\tilde{j}}[r]  \ar@{_{(}->}[d] & H_1(\Gamma,\Z)\ar@{^{(}->}[d] \\
C_1(\Gamma^2,\Z) \ar@{^{(}->}^j[r] & C_1(\Gamma,\Z),
}
\end{equation}
where the vertical maps are the inclusions, $j$ is induced by the inclusion
(\ref{inc}),
and   $\tilde{j}$ denotes the restriction of $j$.
%To prove that  $\tilde{j}$ maps $H_1(\Gamma^2,\Z)$ into $H_1(\Gamma,\Z)$, observe first that
%the image of $j$ is identified with the subspace $K_1\subset C_1(\Gamma, \Z)$ defined
%by $$K_1:=\bigcap_{e\in \Esep}\ker e^* $$ (notation in (\ref{e*})).
%Now,  Remark~\ref{e0sep}  yields that $H_1(\Gamma,\Z)\subset
%K_1$. The contraction map $\Gamma \to \Gamma^2$ induces a bijection between the cycles
%of $\Gamma$ and those of $\Gamma^2$, therefore the image of $\tilde{j}$ lies in
%$H_1(\Gamma,\Z)$. Now, by  lemma~\ref{3lm}(\ref{3lmb}) we get that  $\tilde{j}$ induces an isomorphism
%  $H_1(\Gamma^2,\R)\cong H_1(\Gamma,\R)$, and hence an injection
%$\tilde{j}: H_1(\Gamma^2,\Z)\ha H_1(\Gamma,\Z)$.
 % To prove that $\tilde{j}$ is surjective, it suffices to observe that
%  every $c\in  H_1(\Gamma,\Z)$ is a linear combination of non-separating edges of
%  $\Gamma$, and hence a linear combination of edges of $\Gamma^2$.
%The claim is thus proved.
Part (\ref{Alb2}) follows from the diagram  and the fact that the inclusion
$j$ is compatible with the scalar products $(,)$ on both sides.
\end{proof}
From Proposition~\ref{Alb-eq} we derive that  if $\Gamma^2\equiv_{\rm {cyc}}\Gamma'^2$ then $\Alb(\Gamma)=\Alb (\Gamma')$. Hence the sufficiency in Theorem~\ref{main-thm} is proved.
\end{proof}
In order to prove the other half of the theorem, we need   some preliminaries.
\subsection{The Delaunay decomposition}

Consider the lattice $H_1(\Gamma,\Z)$ inside the real vector space
$H_1(\Gamma,\R)$. Observe that    the scalar
product induced on $C_1(\Gamma,\R)$ by $(,)$  coincides with the Euclidean scalar product.
We denote the norm $\sqrt{(x,x)}$ by $||x||$.

\begin{defi}
\label{Deldef}
For any $\alpha\in H_1(\Gamma,\R)$,  a lattice
element
$x \in H_1(\Gamma,\Z)$ is  called {\it $\alpha$-nearest} if
$$||x-\alpha||={\rm min}\{||y-\alpha||\: : \: y\in H_1(\Gamma,\Z)\}.$$
A {\it Delaunay cell} is defined as  the
closed convex hull of all
%lattice
elements of $H_1(\Gamma,\Z)$
which are $\alpha$-nearest for some fixed $\alpha \in H_1(\Gamma,\R)$.
Together, all the Delaunay cells constitute a locally finite decomposition of $H_1(\Gamma,\R)$
  into  infinitely many  bounded convex polytopes,   called the {\it Delaunay decomposition}
of $\Gamma$, denoted $\Del (\Gamma)$.
%%r<

Let $\Gamma$ and $\Gamma'$ be two graphs. We say that
$\Del(\Gamma)\cong \Del(\Gamma')$ if there exists a
linear isomorphism  
$H_1(\Gamma,\R) \to H_1(\Gamma',\R)$
sending $H_1(\Gamma,\Z)$ into $H_1(\Gamma',\Z)$ and mapping the Delaunay 
cells
of $ \Del(\Gamma)$ isomorphically into the Delaunay cells of $\Del(\Gamma')$.
%%r>
\end{defi}

 \begin{remark}
  \label{Del}
It is well known  that an equivalent,
and for us very useful,
definition  is the following.
The Delaunay decomposition   $\Del (\Gamma)$   is the restriction to
$H_1(\Gamma, \R)$ of the decomposition of $C_1(\Gamma, \R)$
consisting of the standard cubes cut out by all hyperplanes of equation
$e^*=n$ for $e\in E(\Gamma)$ and $n\in \Z$; see \cite[Prop. 5.5]{OS}.
%%r<
These hyperplanes of $H_1(\Gamma, \R)$, having equations $e^*=n$, are called the {\it generating hyperplanes} of the Delaunay decomposition. Notice that an isomorphism $\Del(\Gamma)\cong \Del(\Gamma')$ induces a bijection between the sets of generating hyperplanes.
%%r>

 \end{remark}

\begin{prop}\label{Del-equ}
Let $\Gamma$ and $\Gamma'$ be two
%C connected
graphs.
\begin{enumerate}[(i)]
 \item\label{Del1}
$\Del(\Gamma)$ depends only on $[\Gamma]_{\rm cyc}$.
\item \label{Del2}
$\Del(\Gamma)\cong \Del(\Gamma^3)$ for any choice of  $\Gamma^3$.
\item \label{Del3}
${\rm Del}(\Gamma)\cong  {\rm Del}(\Gamma')$ if and only if
$\Gamma^3\equiv_{\rm cyc}\Gamma'^3$.
\end{enumerate}
\end{prop}
\begin{proof}
It is clear that  the Delaunay decomposition is completely
determined by the inclusion $H_1(\Gamma,\Z)\subset C_1(\Gamma,\Z)$ together with the
basis $E(\Gamma)$ of $C_1(\Gamma,\Z)$ defining the scalar product $(,)$. This proves part (\ref{Del1}).

Let us now prove part (\ref{Del2}). First note that $\Del(\Gamma)\cong\Del(\Gamma^2)$,
as it follows easily from diagram (\ref{diag1}) and  Remark~\ref{e0sep}. We can therefore
assume that $\Gamma$ is $2$-edge connected.
Consider the natural bijection
(cf. (\ref{Smap}))
$$
\begin{aligned}
\psi: \Set  \Gamma &\longrightarrow & E(\Gamma^3) \\
S  &\longmapsto &  e_S \
\end{aligned}
$$
where $e_S$ is the only edge in $S$ which is not contracted by the contraction map $\sigma:\Gamma \to \Gamma^3$.
We can thus define an injection $$
\begin{aligned}
C_1(\Gamma^3,\Z) & \stackrel{\iota}{\longrightarrow} &C_1(\Gamma,\Z)\\
e_S &\longmapsto &\sum_{i=1}^{\#S}e_{S,i},\
\end{aligned}
$$
where for any $S\in \Set \Gamma$ we denote, as in (\ref{Sset}),
$S=\{ e_{S,1},\ldots, e_{S,\# S} \}$. Fix now a totally cyclic  orientation on $\Gamma$
and  the induced orientation on $\Gamma^3$; consider the corresponding spaces $H_1(\Gamma,\Z)$ and $H_1(\Gamma^3,\Z)$.
We claim that the above injection induces a natural diagram
\begin{equation}\label{diag2}
\xymatrix{
H_1(\Gamma^3,\Z) \ar@{^{}->}_{\cong}^{\tilde{\iota}}[r] \ar@{_{(}->}[d] & H_1(\Gamma,\Z)\ar@{^{(}->}[d] \\
C_1(\Gamma^3,\Z) \ar@{^{(}->}^{\iota}[r] & C_1(\Gamma,\Z),
}
\end{equation}
where the vertical maps are the inclusions,
and $\tilde{\iota}$ is the restriction of $\iota$.
Indeed, the image of $\iota$ is clearly   the subset $K_2\subset C_1(\Gamma, \Z)$ defined
by
$$
K_2:=\bigcap_{S\in \Set \Gamma} \bigcap _{i,j=1}^{\#S}\ker (e^*_{S,i}-e^*_{S,j}).
$$
Moreover, by Corollary~\ref{cH}    we get
that $H_1(\Gamma,\Z)\subset K_2$.
On the other hand, the contraction map $\sigma:\Gamma \to \Gamma^3$
induces a bijection between cycles, therefore $H_1(\Gamma^3,\Z)$ maps into
$H_1(\Gamma,\Z)$.
It remains  to prove that $H_1(\Gamma^3,\Z)$ surjects onto $H_1(\Gamma,\Z)$.
We use again  Corollary~\ref{cH} , according to which any $c\in H_1(\Gamma,\Z)$
has the form
$c=\sum_{S\in \Set \Gamma} r_S(c)\sum _{i=1}^{\#S} e_{S,i}$, with $r_S(c)\in \Z$. Hence
$$
c=\iota\bigr(\sum_{S\in \Set \Gamma} r_S(c) e_S\bigl).
$$
At this point  (\ref{Del2}) follows from diagram
(\ref{diag2}) and the fact that, by Corollary~\ref{cH},   $e, f$ belong to the same C1-set if and only if
$e^*_{|H_1(\Gamma,\Z)}=f^*_{|H_1(\Gamma,\Z)}$.

The implication if of part (\ref{Del3}) follows from the previous parts.
In order to prove the other implication, we can assume
that $\Gamma$ and $\Gamma'$ are 3-edge connected.

We claim that, as $\Gamma$ is  3-edge connected,
the functionals $e^*$ restricted to $H_1(\Gamma, \R)$ are all non zero and distinct,  
as $e$ varies in $E(\Gamma)$ (and the same holds for $\Gamma'$ of course).
That $e^*$ is nonzero follows from the fact that $\Esep$ is empty (cf.   \ref{e0sep}).
Let $e\neq f$,  now $\{e\}$ and $\{f\}$ are C1-sets (by \ref{cor3}). By Corollary~\ref{cH}
the restrictions of $e^*$ and $f^*$ to $H_1(\Gamma, \R)$ are different. The claim is proved.

The claim means that the intersections of the hyperplanes $\{e^*=0\}_{e\in E(\Gamma)}$ with $H_1(\Gamma, \R)$ are all proper and
distinct, and similarly for $\Gamma'$. 
Now, an isomorphism
$\Del(\Gamma)\cong \Del(\Gamma')$ induces a bijection between the sets of generating hyperplanes passing through the origin; hence, by the claim, we get a bijection
$E(\Gamma)\cong E(\Gamma')$ (which extends to an isomorphism
$C_1(\Gamma, \Z)\cong C_1(\Gamma',\Z)$).

To conclude, we now use a  basic fact from  graph theory (see for example \cite[Sect. 5.1]{Oxl}
 or   \cite[Thm. 3.11]{alex}), according to which
 the $0$-skeleton of the
hyperplane arrangement $\{e^*=n,\  \  e\in E(\Gamma), n\in \Z\}$ in $H_1(\Gamma,\R)$
is the lattice $H_1(\Gamma,\Z)$ itself. Therefore, we deduce that the above
bijection $E(\Gamma)\cong E(\Gamma')$ induces an isomorphism
$H_1(\Gamma,\Z)\cong  H_1(\Gamma',\Z)$,
from which we conclude that $ \Gamma\equiv_{\rm cyc} \Gamma'$.
\end{proof}

\begin{remark}\label{Artam}
A special case of   Proposition~\ref{Del-equ} has been proved by Artamkin
using a different language.
In  \cite{art}, he associates to a graph $\Gamma$ a convex integral polytope
$\Delta(\Gamma)$ in $H_1(\Gamma,\R)$, called the ``simple cycle polytope",
and he proves that a 3-connected graph $\Gamma$ is uniquely determined by
$\Delta(\Gamma)$ (see \cite[Thm. 1]{art}).   $\Delta(\Gamma)$
turns out to be   the union of the maximal dimensional Delaunay cells that have a vertex in
the origin; hence   knowing $\Delta(\Gamma)$ is equivalent to knowing
$\Del(\Gamma)$. Using this observation and \ref{3v}, \cite[Thm. 1]{art}
is equivalent to Proposition \ref{Del-equ}(\ref{Del3}), provided that $\Gamma$ and $\Gamma'$ are 3-connected.
\end{remark}

\subsection{Proof of Theorem \ref{main-thm}: necessity.}
\begin{proof}
Assume   that $\Alb(\Gamma)\cong\Alb(\Gamma')$. Using
\ref{Alb-eq}, we can assume that $\Gamma$ and $\Gamma'$ are
$2$-edge connected.
We fix a totally cyclic orientation on them.
Since the Delaunay decomposition is completely
determined by $(H_1(\Gamma,\Z);(,))$, i.e. by $\Alb(\Gamma)$
(see \cite[Sec. 2]{KS}), we have $\Del(\Gamma)\cong\Del(\Gamma')$. We can thus apply Proposition \ref{Del-equ}(\ref{Del3}),
getting that  $\Gamma^3\equiv_{\rm cyc} \Gamma'^3$.
Therefore, by Proposition~\ref{lift-cyc}(\ref{lift2}) there is a natural bijection,
\begin{equation}
\label{beta}
\Set \Gamma \stackrel{\beta}{\la}\Set \Gamma' ;\  \  \  S\mapsto S':=\beta(S).
\end{equation}
To prove the theorem it suffices to show  that $\beta$ preserves the cardinalities.
In fact by Proposition~\ref{lift-cyc}(\ref{lift2}), this implies that
$ \Gamma\equiv_{\rm cyc}  \Gamma'$.

First, note that
by hypothesis there is an isomorphism, denoted
\begin{equation}
\label{Hiso}
H_1(\Gamma,\Z)\stackrel{\cong}{\la}H_1(\Gamma',\Z);\  \  \  c\mapsto c'
\end{equation}
such that $(c_1,c_2)=(c'_1,c'_2)$  for all $c_i\in H_1(\Gamma,\Z)$.
Pick $c\in H_1(\Gamma,\Z)$; by Corollary~\ref{cH}
we can write
$c=\sum_{S\in \Set \Gamma} r_S(c)\sum _{i=1}^{\#S} e_{S,i}$,
with $r_S(c)\in \Z$;
hence we can define (consistently with \ref{decdelta})
the set
$$
\Set_c\Gamma:=\{S\in \Set \Gamma: r_S(c)\neq 0\}.
$$
%We call $\Set_c\Gamma$ the   C1-decomposition of   $c$.
We claim that for every $S\in \Set \Gamma$ and every
$c\in  H_1(\Gamma, \Z)$ we have
\begin{equation}
\label{divr}r_{S'}(c')=u(S)r_S(c) ,\  \  \  u(S):=\pm 1;
\end{equation}
in particular, \begin{equation}
\label{div}
S\in \Set_c \Gamma \Leftrightarrow S'\in\Set_{c'} \Gamma' .
\end{equation}
To prove the claim, consider the affine function $f^n_S:C_1(\Gamma, \Z)\to \Z$ defined as
$$
f^n_S:=e_S^*-n,\  \  \  n\in \Z.
$$
By what we said before we have
$$
r_S(c)=n \Leftrightarrow  c\in \ker f^n_S.
$$
Observe that the bijections (\ref{Hiso}) and (\ref{beta}) are compatible with one another.
In other words, for every $c\in H_1(\Gamma, \Z)$, the set $\Set_c \Gamma$ is mapped to
$\Set_{c'} \Gamma '$ by $\beta$.
Therefore
the isomorphism between $\Alb(\Gamma)$ and $\Alb (\Gamma ')$ induces
a bijection between the hyperplanes generating $\Del (\Gamma)$ and those
generating  $\Del (\Gamma ')$
%%r< 
such that $f_{S}^n$ is mapped either to $f_{S'}^n$ or to $f_{S'}^{-n}$
(see Remark~\ref{Del}).
%%r>
%. This bijection maps $f_S^n$ either to $f_{S'}^n$ or to $f_{S'}^{-n}$.
So,  the claim is proved.

To ease the notation, in the sequel  for any $S\in \Set \Gamma$ we denote
$$
e(S):=\sum _{i=1}^{\#S} e_{S,i}.
 $$
 Moreover, if $S\in  \Set_c\Gamma$ for some $c$, we denote
\begin{equation}
\label{rc}
\lambda(c-S):=
\sum _{T\in \Set_c\Gamma \smallsetminus \{S\}}\#  T.
\end{equation}
Observe that for any   cycle $\Delta\subset \Gamma$ of length $\lambda$
and any     $c:=\sum_{e\in E(\Delta)}\pm e\in H_1(\Gamma,\Z)$
(such a $c$ exists for a suitable choice of signs),
we have $\Set_{\Delta}\Gamma=\Set_{c}\Gamma$ and
$$
\lambda=||c||^2=\sum_{S\in \Set_{\Delta}\Gamma}\#S=\#S+\lambda (c-S)
$$
for any $S\in  \Set_{\Delta}\Gamma$.

We shall now prove
 that  the map (\ref{beta}) preserves   cardinalities.
By contradiction, suppose there exists $S\in \Set \Gamma$ such that
\begin{equation}
\label{cardS}
\# S>\# S'.
\end{equation}
By Lemma \ref{intS}, we can find two cycles $\Delta_1$ and $\Delta_2$ of $\Gamma$ such that
$S=E(\Delta_1)\cap E(\Delta_2)$.
For $i=1,2$, there exists an element $c_i\in H_1(\Gamma,\Z)$ given by the formula
$$
c_i= e(S)+\sum_{T\in \Set_{c_i}\Gamma  \smallsetminus \{S\}}\pm e(T)
$$
(by Corollary~\ref{cH}).
The sign before $e(T)$ will play no role, so we can ignore it.

Suppose to fix ideas  that $u(S)=1$ in (\ref{divr}). The case $u(S)=-1$ is treated   in a trivially analogous way
(we omit the details).
By (\ref{divr}),   we have
$$
c_i'= e(S')+\sum_{T\in \Set_{c_i}\Gamma  \smallsetminus \{S\}}\pm e(T')
$$
%%r<
(recall that $T$ determines $T'$ uniquely).
%%r>
Therefore, as $||c_i||^2=(c_i,c_i)=(c_i',c_i')= ||c_i'||^2$, using notation (\ref{rc})
$$
||c_i||^2=\#S+\lambda(c_i-S) = ||c_i'||^2=\#S'+\lambda(c_i'-S').
$$
 By (\ref{cardS}) we get for $i=1,2$
 \begin{equation}
\label{cardT}
\lambda(c_i-S)<  \lambda(c_i'-S').
\end{equation}
Now,  $c_1-c_2$ lies, of course, in $H_1(\Gamma, \Z)$; we have
$$
c_1-c_2= \sum_{T\in \Set_{c_1}\Gamma  \smallsetminus \{S\}}\pm e(T)-
\sum_{U\in \Set_{c_2}\Gamma  \smallsetminus \{S\}}\pm e(U).
$$
Since $ \Set_{c_1}\Gamma\cap \Set_{c_2}\Gamma =\{S\}$,    we have
\begin{equation}\label{diff1}
||c_1-c_2||^2=\lambda(c_1-S)+\lambda(c_2-S).
\end{equation}
Arguing in the same way for $c_1'-c_2'$, we get
\begin{equation}\label{diff2}
||c_1'-c_2'||^2=\lambda(c_1'-S')+\lambda(c_2'-S').
\end{equation}
Therefore, using   (\ref{diff2}), (\ref{cardT}) and (\ref{diff1})
$$
%\lambda(c_1-S)+\lambda(c_2-S)=||c_1-c_2||^2=
||c_1'-c_2'||^2=\lambda(c_1'-S')+\lambda(c_2'-S')>
\lambda(c_1-S)+\lambda(c_2-S)=||c_1-c_2||^2
$$
which contradicts the fact that the isomorphism (\ref{Hiso}) preserves the scalar
products.
\end{proof}
%%r<
In the proof we applied the next Lemma, which will be used again later on. 
%%r>
\begin{lemma}
\label{intS}
Let $S\in \Set \Gamma$. For every cycle $\Delta\subset \Gamma$ such that $S\subset E(\Delta)$ there exists a cycle $\hat{\Delta}\subset \Gamma$
such that $S=E(\Delta)\cap E(\hat{\Delta})$.
\end{lemma}
\begin{proof}
It is clear that it suffices to assume $\Gamma$ free from separating edges.
We begin by reducing to the case $\#S=1$. Choose an edge $e\in S$ and
consider the map
$\Gamma \to \ov{\Gamma}$ contracting all edges of $S$ but $e$. Then $\sigma$ induces a bijection between the cycles of $\Gamma$ and those of $\ov{\Gamma}$, and it is clear that
if the statement holds on $\ov{\Gamma}$ it also holds on ${\Gamma}$.

So, let $S=\{e\}$ and let $\Delta$ be a cycle containing $e$.
We shall exhibit an iterated procedure which yields, at its $i$-th step, a cycle $\Delta_{i}$
containing $e$ and such that $\# E(\Delta)\cap  E(\Delta_i)$ decreases at each step.
Set $\Delta_1=\Delta$ and $S_1:=S=\{e\}$; if $\Delta$ has length 1
  we   take $\hat{\Delta}=\Delta$ and we are done.
So, suppose $\#E(\Delta)  \geq 2$; we can decompose $E(\Delta)$ as a disjoint union of C1-sets
$E(\Delta)=\{e\}\cup S_2\cup\ldots\cup S_{h}$, with $S_i\in \Set \Gamma$
(cf. Remark~\ref{decdelta}).
For the second step
consider $\Gamma_2:=\Gamma \smallsetminus S_2$; then $\Gamma_2$ has no   separating edges, therefore there exists a cycle $\Delta_2\subset \Gamma_2$ containing $e$.
Obviously $\Delta_2$ does not contain $S_2$, hence
$\# E(\Delta)\cap  E(\Delta_2)< \# E(\Delta)\cap  E(\Delta_1)$.
If $\Delta_2$ does not contain
any other edge of $\Delta$
 we take $\Delta_2=\hat{\Delta}$ and we are done. Otherwise we repeat the process within
 $\Gamma_2$. Namely, we have $E(\Delta_2)=\{e\}\cup S_2^2\cup\ldots\cup S^2_{h}$, with $S_i^2\in \Set \Gamma_2$, set $\Gamma_3:=\Gamma_2\smallsetminus S_2^2$.
 There exists a cycle $\Delta_3\subset \Gamma_3$ containing $e$, and it is clear that
 $\# E(\Delta)\cap  E(\Delta_3)< \# E(\Delta)\cap  E(\Delta_2)$.

Obviously this   process must terminate after, say $m$, steps,
 when we necessarily  have
$E(\Delta)\cap  E(\Delta_m)=\{e\}$.
\end{proof}

\section{Torelli theorem for metric graphs and tropical curves}
\label{trp}
In this section we apply the methods and results of the previous part to study the Torelli problem for
tropical curves. We refer to \cite{MIK3}, or to \cite{MZ}, for  details about  the theory of tropical curves and their Jacobians.

\subsection{Tropical curves, metric graphs and
associated tori}\label{Sec4.1}

Let $C$ be a compact tropical curve;
$C$ is endowed with
a Jacobian variety, $\Jac (C)$, which is a principally polarized tropical Abelian variety
 (see \cite[Sec. 5]{MZ} and \cite[Sect 5.2]{MIK3});
 we shall denote $(\Jac (C), \Theta_C)$
 the principally polarized Jacobian of $C$, where $\Theta_C$ denotes the principal polarization
 (see Remark~\ref{albjac} below).  Observe that  two tropically equivalent curves have isomorphic Jacobians.
As we stated in  the introduction,  we want to study the following

\

\noindent
{\bf{Problem.}}
{\it {For which compact tropical curves  $C$ and $C'$  there is an isomorphism $(\Jac (C), \Theta_C) \cong
(\Jac (C'), \Theta_{C'})$?}}

\

\noindent
We   will     answer     this question
in  Theorem \ref{main-thmt}.

As we already mentioned, the connection with the earlier sections of this paper comes from  a result of G. Mikhalkin and I. Zharkov, establishing that tropical curves are closely related to metric graphs.\begin{defi}
\label{mg}
A metric graph $(\Gamma, l)$ is a finite graph $\Gamma$ endowed with a function
$l:E(\Gamma)\to \R_{>0}$
called the {\it length function}.
\end{defi}
\begin{remark}
\label{leaves}
Our  definition of metric graph coincides with that of \cite{MZ}
only if the graph has valence at least 2. The difference occurs in the length function,
whereas the graph is the same.
More precisely,
 the definition of length function used in \cite{MZ} differs from ours, 
as it assigns the value $+\infty$ to every edge adjacent to a vertex of valence 1;
 such edges are called {\it leaves}.
With this definition, metric graphs are in bijection with  tropical curves.

To avoid trivial cases,  we shall always assume that our tropical curves have genus at least 2.
Under this assumption, by  \cite[Prop. 3.6]{MZ},  there is a one to one correspondence
between  tropical equivalence classes of compact tropical curves    and metric graphs with
valence at least 3
(i.e such that every vertex  has at least three incident edges).

Therefore, from now on,  we identify compact tropical curves, up to tropical 
equivalence, with metric graphs of valence at least 3.

\end{remark}
\begin{remark}
\label{termtrop}
Since to every compact tropical curve $C$ we associate a
unique finite graph $\Gamma$, we will use for $C$ the graph theoretic terminology.
In particular, we shall say that $C$ is $k$-connected if so is $\Gamma$.
\end{remark}
Given a metric graph $(\Gamma, l)$,  we define the scalar product $(,)_l$ on $C_1(\Gamma,\R)$
as follows
$$(e,e')_l=\begin{cases}
l(e)& \text{ if } e=e', \\
0 & \text{ otherwise. }
\end{cases}
$$
In analogy with Definitions~\ref{alb},
\ref{cyc-equiv} and \ref{3-conn}
  we shall define the Albanese torus, the cyclic equivalence, and  the $3$-edge connectivization
   for metric graphs.

\begin{defi}\label{Alb-torusl}
The Albanese torus $\Alb(\Gamma,l)$ of the metric graph $(\Gamma,l)$ is
$$
\Alb(\Gamma,l):=\bigr(H_1(\Gamma,\R)/H_1(\Gamma,\Z); (,)_l\bigl)
$$
(with the flat metric derived from the scalar product $(,)_l$).
\end{defi}
\begin{remark}
\label{albjac}
By \cite[Sect. 6.1 p. 218]{MZ} we can naturally identify
 $(\Jac (C),\Theta_C)$ with the Albanese torus ${\rm Alb}(\Gamma,l)$.
\end{remark}

\begin{defi}\label{cyc-equivl}
Let  $(\Gamma,l)$ and $(\Gamma',l')$ be two metric graphs.
We say that  $(\Gamma,l)$ and $(\Gamma',l')$ are {\it cyclically equivalent}.
and we write $(\Gamma,l)\equiv_{\rm cyc} (\Gamma',l')$, if there exists a cyclic bijection
$\e:E(\Gamma)\to E(\Gamma')$ such that $l(e)=l'(\epsilon(e))$ for all $e\in E(\Gamma)$.
The cyclic equivalence class of $(\Gamma,l)$ will be  denoted by $[(\Gamma,l)]_{\rm cyc}$.
\end{defi}

\begin{defi}\label{3-connl}
A {\it $3$-edge connectivization} of a metric graph $(\Gamma,l)$ is a
  metric graph $(\Gamma^3, l^3)$, where $\Gamma^3$ is a
$3$-edge connectivization of $\Gamma$, and $l^3$ is the length function  defined
as follows, $$l^3(e_S)=\sum_{e\in \psi^{-1}(e_S)} l(e)=\sum_{e\in  S } l(e)
$$
where, with the notation of (\ref{Smap}),
 $\psi:\Set \Gamma \to E(\Gamma^3)$ is the natural  bijection mapping $S$ to $e_S$.
\end{defi}
\begin{remark}
\label{3ecl}
Using lemma~\ref{3lm}(\ref{3lmcyc}) we see that all the $3$-edge connectivizations of a metric graph $(\Gamma,l)$
are cyclically equivalent.
%Hence the class $[(\Gamma^3, l^3)]_{\rm cyc}$  depends only on $[(\Gamma, l)]_{\rm cyc}$.

Observe also that $[(\Gamma^3, l^3)]_{\rm cyc}$  is completely independent of the separating edges of $\Gamma$, and on the value that $l$ takes on them.
Therefore, $(\Gamma^3, l^3)$ is well defined also if   $l$ takes value
$+\infty$ on the leaves of $\Gamma$.
This enables us to
define $[(\Gamma^3, l^3)]_{\rm cyc}$ for a graph $(\Gamma, l)$, metric in the sense of \cite{MZ},
associated to a tropical curve $C$ (see Remark~\ref{leaves}).

Consistently with Remark~\ref{3ec}, we call
$[(\Gamma^3, l^3)]_{\rm cyc}$  the {\it 3-edge connected class of $C$}.
With this terminology, we state the main result of this section:
\end{remark}

\begin{thm}\label{main-thmt}
Let $C$ and $C'$ be  compact tropical curves.
Then   $(\Jac C, \Theta_C)\cong (\Jac C', \Theta_{C'})$
if and only if $C$ and $C'$ have the same 3-edge connected class.

Suppose that   $C$ is 3-connected. Then $(\Jac C, \Theta_C)\cong (\Jac C', \Theta_{C'})$
if and only if $C$ and $C'$ are
tropically  equivalent.
\end{thm}
\begin{proof}
The first statement is a straightforward consequence of  the next Theorem~\ref{main-thml}.

We let $\Gamma$ and $\Gamma '$ be the metric graphs associated to $C$ and $C'$ respectively.
Suppose now that $C$ is 3-connected. This means (cf. \ref{termtrop}) that the associated
graph is 3-connected (and hence 3-edge connected).
By the previous part $\Gamma =\Gamma^3\cong \Gamma'^3$.
%%r<
Recall that, by convention (cf. Remark~\ref{leaves}),  the graph $\Gamma'$ has valence at least $3$.
%%r>
To finish the proof it suffices  to show that the map $\sigma :\Gamma'\to \Gamma'^3$ is the identity map;
to do that we will use the fact that $\Gamma '^3$ is 3-connected, as $\Gamma$ is.

Suppose  $\sigma $ contracts a separating edge  $e$ of $\Gamma '$; observe that the two vertices adjacent to $e$
are both separating vertices for $\Gamma '$, because $\Gamma '$ has no vertices of  valence 1.
But then $\sigma(e)$ would be a separating vertex of $\Gamma'^3$,
which is impossible.

If $\sigma $ contracts one edge  of a separating pair,
  arguing in a similar way we obtain that
  $\Gamma '$ has a separating pair of vertices which is mapped by $\sigma$
  to a  separating pair of vertices  of $ \Gamma'^3$, which is impossible.
Therefore $\sigma$ is the identity and we are done.
\end{proof}

\begin{thm}\label{main-thml}
Let $(\Gamma,l)$ and $(\Gamma',l')$ be two metric graphs.
Then   $\Alb(\Gamma,l)\cong \Alb(\Gamma',l')$
if and only if $ [(\Gamma^3,l^3)]_{\rm cyc}=[(\Gamma'^3, l'^3)]_{\rm cyc}$.
\end{thm}

\subsection{Proof of the Torelli theorem for metric graphs}
The proof of Theorem~\ref{main-thml} follows the same steps as the proof of Theorem~\ref{main-thm}.
The ``if"  part    follows easily  from the following

\begin{prop}\label{Alb-eql}
Let $(\Gamma,l)$ be a metric graph.
\begin{enumerate}[(i)]
 \item \label{Alb1l}
$\Alb(\Gamma,l)$ depends only on $[(\Gamma,l)]_{\rm cyc}$.
\item \label{Alb2l}
$\Alb(\Gamma,l)\cong \Alb(\Gamma^3,l^3)$ for any $3$-edge connectivization of $(\Gamma,l)$.
\end{enumerate}
\end{prop}
\begin{proof}
Part (\ref{Alb1l}) follows from the fact that
$(H_1(\Gamma,\Z);(,)_l)$ is defined entirely in terms
of the inclusion $H_1(\Gamma, \Z)\subset C_1(\Gamma, \Z)$ and of the values of
$(,)_l$ on the orthogonal basis $E(\Gamma)$
of $C_1(\Gamma,\Z)$, all of which is clearly invariant by cyclic equivalence.

To prove part (\ref{Alb2l}) we use the proof of Proposition~\ref{Del-equ}(\ref{Del2}), to which we now refer
for the notation.

Consider the diagram (\ref{diag2}). The point is  that
the inclusion $\iota$  is compatible with the scalar
product $(,)_l$ on the right and the scalar product $(,)_{l^3}$ on the left. More precisely, for every edge $e_S$ of $\Gamma^3$
(so that $S\in \Set \Gamma$) we have (by definition of $l^3$)
$$
(e_S,e_S)_{l^3}=l^3(e_S)=\sum_{e\in S}l(e)= \bigl(\sum_{i=1}^{\#S}e_{S,i},\sum_{i=1}^{\#S}e_{S,i}\bigr)_l
=\bigl(\iota(e_S),\iota(e_S)\bigr)_l.
$$
On the other hand if $T\in \Set \Gamma$ with $T\neq S$ we have
$0=(e_S,e_T)_{l^3}=(\iota(e_S),\iota(e_T))_l$
(since $S\cap T=\emptyset$).

Therefore (\ref{Alb2l}) is proved, and with it the sufficiency part of Theorem~\ref{main-thml}.
\end{proof}
To prove the opposite implication of Theorem~\ref{main-thml},
we need the following
\begin{defi}\label{del-metr}
The Delaunay decomposition $\Del(\Gamma,l)$ associated to the metric graph
$(\Gamma,l)$ is the Delaunay decomposition (cf. Definition~\ref{Deldef}) associated to the scalar product
$(,)_l$ on $H_1(\Gamma,\R)$ with respect to the lattice $H_1(\Gamma,\Z)$.
\end{defi}

\begin{lemma}\label{Del-lemma}
Let $(\Gamma,l)$ be a metric graph. Then
\begin{enumerate}[(i)]
\item
\label{Di}
$\Del(\Gamma,l)$ is determined by $\Alb(\Gamma,l)$.
\item
\label{Dii} $\Del(\Gamma,l)=\Del(\Gamma)$.
\end{enumerate}
\end{lemma}
\begin{proof}
Clearly, $\Del(\Gamma,l)$ is determined by the lattice $H_1(\Gamma,\Z)\subset H_1(\Gamma,\R)$
and the scalar product $(,)_l$, and therefore by $\Alb(\Gamma,l)$. This shows part (\ref{Di}).

Part (\ref{Dii}) follows from a well-known Theorem of Mumford,  \cite[Thm 18.2]{nam}.
\end{proof}

\noindent
{\it Proof of Theorem~\ref{main-thml}:  necessity.}
Suppose   that $\Alb(\Gamma,l)\cong\Alb(\Gamma',l')$.  By Lemma~\ref{Del-lemma},
 $(\Gamma,l)$ and $(\Gamma',l')$ have the same Delaunay decompositions
and $\Del (\Gamma) = \Del (\Gamma ')$.
We can assume that $\Gamma$ and $\Gamma'$ are
$3$-edge connected.
By Proposition~\ref{Del-equ}(\ref{Del3}), we have  $\Gamma \equiv_{\rm cyc} \Gamma'$.
Denote
\begin{equation}
\label{epsl}
E( \Gamma) \stackrel{\epsilon}{\la} E(\Gamma') ;\  \  \  e\mapsto e':=\epsilon(e)
\end{equation}
a cyclic bijection.
It remains  to prove that $l(e)= l'(e')$ for every $e\in E( \Gamma)$.
 We will proceed in strict analogy with the proof of the necessity  of Theorem~\ref{main-thm}.

First, note that
  there is an isomorphism, denoted
\begin{equation}
\label{Hisol}
H_1(\Gamma,\Z)\stackrel{\cong}{\la}H_1(\Gamma',\Z);\  \  \  c\mapsto c'
\end{equation}
such that $(c_1,c_2)_l=(c'_1,c'_2)_{l'}$  for all $c_i\in H_1(\Gamma,\Z)$.
Pick $c\in H_1(\Gamma,\Z)$ and write
$c=\sum_{e\in E( \Gamma)} r_e(c)  e $,
with $r_e(c)\in \Z$;
similarly $c'=\sum_{e'\in E( \Gamma')} r_{e'}(c')  e' $
with $r_{e'}(c')\in \Z$. We claim that for every $e\in E(\Gamma)$ and every
$c\in  H_1(\Gamma, \Z)$ we have
\begin{equation}
\label{divrl}r_{e'}(c')=u(e) r_e(c),\  \  \  u(e):=\pm 1.
\end{equation}
To prove the claim,  notice that  $
r_e(c)=n \Leftrightarrow  e^*(c)=n.
$
On the other hand, the isomorphism between   $\Del (\Gamma)$ and    $\Del (\Gamma ')$
  maps the hyperplane of equation $e^*=n$ either to $ e'^*=n$ or to $ e'^*=-n$.
So,  the claim is proved.
Now
  define  $ E_c(\Gamma):=\{e\in E(\Gamma): r_e(c)\neq 0\}.
$ For any  $c\in  H_1(\Gamma, \Z)$  and $e\in  E_c(\Gamma)$  we shall denote
\begin{equation}
\label{rcl}
\lambda(c-e):=
\sum _{f\in  E_c(\Gamma) \smallsetminus \{e\}}l(f).
\end{equation}

We can now prove
 that  the map (\ref{epsl}) preserves the lengths, i.e. that
 $l(e)=l'(e')$ for every $e\in E(\Gamma)$.
By contradiction, suppose there exists an edge $e$ of $ \Gamma$ such that
\begin{equation}
\label{cardSl}
l(e)>l'(e').
\end{equation}
By Lemma \ref{intS}, there exist two cycles $\Delta_1$ and $\Delta_2$ of $\Gamma$ such that
$\{e\}=E(\Delta_1)\cap E(\Delta_2)$. As in
the proof of   Theorem
\ref{main-thm},
consider   $c_1 $ and $c_2 $
in $H_1(\Gamma,\Z)$ associated to the above two cycles (so that $ E_{c_i}(\Gamma)=\Set_{\Delta_i}\Gamma$ for $i=1, 2$):
$$
c_i= e+\sum_{f\in  E_{c_i}(\Gamma)  \smallsetminus \{e\}}\pm f.
$$
The sign before $f$ will  play no role,   hence we are free to ignore it.
Suppose   that $u(e)=1$ (the case $u(e)=-1$ is treated   similarly).
By (\ref{divrl}) we have$$
c_i'= e'+\sum_{f\in  E_{c_i}(\Gamma)  \smallsetminus \{e\}}\pm f'.
$$
Therefore, as $||c_i||^2:=(c_i,c_i)_l=(c_i',c_i')_{l'}=: ||c_i'||^2$,
using notation (\ref{rcl}) we have
$$
||c_i||^2=l(e)+\lambda(c_i-e) =||c_i'||^2=l(e')+\lambda(c_i'-e').$$
 By (\ref{cardSl}) we get
 \begin{equation}
\label{cardTl}
\lambda(c_i-e)<  \lambda(c_i'-e').
\end{equation}
Now consider $c_1-c_2\in H_1(\Gamma, \Z)$. We have
$$
c_1-c_2= \sum_{f\in  E_{c_1}(\Gamma)  \smallsetminus \{e\}}\pm f-
\sum_{g\in E_{c_2}(\Gamma)  \smallsetminus \{e\}}\pm g.
$$
Since $ E_{c_1}(\Gamma)\cap E_{ {c_2}}(\Gamma) =\{e\}$,   we have
$||c_1-c_2||^2=\lambda(c_1-e)+\lambda(c_2-e).$
Arguing similarly for $c_1'-c_2'$, we get
$||c_1'-c_2'||^2=\lambda(c_1'-e')+\lambda(c_2'-e').$
Hence, by (\ref{cardTl})
$$
||c_1'-c_2'||^2=\lambda(c_1'-e')+\lambda(c_2'-e')>\lambda(c_1-e)+\lambda(c_2-e)=
||c_1-c_2||^2,
$$
contradicting the fact that  (\ref{Hisol}) preserves the scalar products.
$\qed$

\section{Further characterizations of graphs}
\label{pos}
The Torelli theorems proved in the previous sections are based on the notion of 3-edge connected class,
$[\Gamma^3]_{\rm{cyc}}$, of a graph $\Gamma$.
The aim of this section,
whose main result is Theorem~\ref{final-thm},
is to give some other characterizations of $[\Gamma^3]_{\rm{cyc}}$.

\subsection{The poset $\SP_{\Gamma}$}
\label{op}

\begin{defi} Let $\Gamma$ be a graph.
The poset  $\SP_{\Gamma}$ is the set
of all the subsets $S\subset E(\Gamma)$ such that the
%spanning
subgraph
$\Gamma\smallsetminus S$ is free from separating
edges, endowed with the following partial order:
$$ S\geq T \Longleftrightarrow
% \Gamma\smallsetminus S\supseteq \Gamma\smallsetminus T \Longleftrightarrow
S\subseteq T .$$
\end{defi}
\begin{remark}
\label{SPsep}
It is clear that for every $S\in \SP_{\Gamma}$ we have $\Esep \subset S$.
Therefore any map $\sigma:\Gamma \to \ov{\Gamma}$ contracting some separating edges
of $\Gamma$
induces a bijection of posets $\SP_{\Gamma} \la \SP_{\overline{\Gamma}}$.
\end{remark}
We will use  some notions and facts  from graph theory.
\begin{defi}\cite[Sect. 2.3]{Oxl}
 \label{coma}
The {\it cographic matroid} $M^*(\Gamma)$ of $\Gamma$ is the matroid of collections of linearly independent vectors
among the collections of vectors $\{e^*\: : \: e\in E(\Gamma)\}$ of
$H_1(\Gamma, \R)^*$.
\end{defi}
\begin{remark}
\label{mator}
It is well known that  the cographic matroid $M(\Gamma)$ is independent of the choice of
the orientation of $\Gamma$ used to define $H_1(\Gamma,\Z)\subset C_1(\Gamma,\Z)$.
\end{remark}

\begin{thm}\cite[Sect. 5.3]{Oxl}\label{matroid}
$M^*(\Gamma)\cong M^*(\Gamma')$ if and only if $\Gamma\equiv_{\rm cyc} \Gamma'$.
\end{thm}

We are going to show that the poset $\SP_{\Gamma}$ is determined by   $M^*(\Gamma)$.
Before doing that we recall the notion of a flat  of the matroid  $M^*(\Gamma)$
(see for example \cite[Sec. 1.7]{Oxl}).
First, for any $S=\{ e_{S,1},\ldots, e_{S,\# S} \}\subset E(\Gamma)$
we denote by
$$
\langle S^*\rangle={\rm span}( e_{S,1}^*,\ldots, e_{S,\# S}^*)\subset H_1(\Gamma,\Z)^*.
$$
We say that $S$ is a  \emph{flat} of $M^*(\Gamma)$ if for every $e\in E(\Gamma)\smallsetminus S$
we have
$$\dim\langle S^*\rangle< \dim {\rm span}(S^*, e^*)=\dim {\rm span}( e_{S,1}^*,\ldots, e_{S,\# S}^*,e^*).
$$
%We endow the set of flats with the structure of a poset with respect to the inverse
%inclusion of flats.

\begin{lemma}\label{flats}
$\SP_{\Gamma}$ is the set of flats of the matroid $M^*(\Gamma)$.
\end{lemma}
\begin{proof}
Given any subset $T\subset E(\Gamma)$,   its closure ${\rm cl}(T)$ is defined as the subset
of $E(\Gamma)$ formed by all the   $e\in E(\Gamma)$ such that
$e^*\in {\rm span}_{f\in T}( f^*)$. It is clear that
$T\subset E(\Gamma)$ is a flat if and only if $T={\rm cl}(T)$.

%We need to characterize when $e\in {\rm cl}(T)$.
%Choosing compatible $orientations of $\Gamma$ and $\Gamma\smallsetminus T$,
We have the following commutative
diagram
\begin{equation*}\xymatrix{
H_1(\Gamma\smallsetminus T,\Z)\ar@{^{(}->}[r] \ar@{^{(}->}[d]& C_1(\Gamma\smallsetminus T,\Z)
\ar@{^{(}->}[d]\\
H_1(\Gamma,\Z) \ar@{^{(}->}[r] & C_1(\Gamma,\Z).
}\end{equation*}
The left vertical injective map induces a surjective map $H_1(\Gamma,\Z)^*
\twoheadrightarrow H_1(\Gamma\smallsetminus T,\Z)^*$ whose kernel is equal to
${\rm span}_{f\in T}(f^*)$. Therefore $e\in {\rm cl}(T)$ if and only
the  image, $[e^*]$, of $e^*$ under the above surjection is zero. If $e\not\in T$ then $[e^*]=0$
if and only if $e$ does not belong to any cycle of $\Gamma\smallsetminus T$, if and only if
$e$ is a separating edge of $\Gamma\smallsetminus T$. This shows that $T={\rm cl}(T)$ if and only
$\Gamma\smallsetminus T$ does not have separating edges, in other words
$T$ is a flat  if and only
 $T\in \SP_{\Gamma}$.
\end{proof}

\begin{remark}\label{arrang}
It is easy to see, from the above definitions, that the poset of flats of $M^*(\Gamma)$
is isomorphic to the poset of intersections of the arrangement of hyperplanes
$\{e^*=0\}_{e\in E(\Gamma)\smallsetminus E(\Gamma)_{\rm sep}}.$
\end{remark}

Recall that a poset $(P,\leq)$ is called {\it graded} if it is has a monotone  function
$\rho:P\to \N$, called the {\it rank function},  such that
if $x$ covers $y$ (i.e. $y\lneq  x$ and there does
not exist a $z$ such that $y\lneq z\lneq x$)
then $\rho(x)=\rho(y)+1$.
If our poset has a minimum element $\un{0}$, we say that it is bounded from below. If this is the case
 $(P,\leq)$ is graded if and only if for every element $x\in P$ all the maximal
chains from $\un{0}$ to $x$ have the same length. We can define a rank function
$\rho:P\to \N$ by setting $\rho(x)$  equal to the length of any chain
from $\un{0}$ to $x$. This is the unique rank function on $(P,\leq)$ such that
$\rho(\un{0})=0$ and we call it the {\it normalized rank function}.

\begin{cor}
\label{rank}
The poset $\SP_{\Gamma}$ is a graded poset with minimum element equal to $E(\Gamma)$
and   normalized rank function   given by
$S\mapsto b_1(\Gamma\smallsetminus S)$.
\end{cor}
\begin{proof}
(It is well-known, see \cite[Thm. 1.7.5]{Oxl}, that the poset of flats of a matroid is a geometric lattice; hence in particular a graded poset.)
The minimum element is clearly $E(\Gamma)$ and the length of a chain
in $\SP_{\Gamma}$ from $E(\Gamma)$ to $S$ is exactly equal to the number of independent
cycles in $\Gamma\smallsetminus S$, that is to $b_1(\Gamma\smallsetminus S)$.
\end{proof}
\begin{remark}
\label{codim}
We like to think of the number $b_1(\Gamma\smallsetminus S)$ as the {\it codimension} of the
set $S\in \SP_{\Gamma}$.
If $\Esep=\emptyset$
(which is a harmless assumption, by remark~\ref{SPsep}),
then $\Set \Gamma \subset \SP_{\Gamma}$, and we have that
 $S$ has codimension 1 if and only if $S$ is a C1-set. (cf. \ref{C1}).
\end{remark}

\begin{lemma}\label{supp-equiv}
Let $\Gamma$ and $\Gamma'$ two
%C connected
graphs. For any choice of $\Gamma^3$ and $\Gamma'^3$ we have:
\begin{enumerate}
\item[(i)] $\SP_{\Gamma}\cong \SP_{\Gamma^3}$  (as posets).
\item[(ii)] $\SP_{\Gamma}\cong \SP_{\Gamma'}$ if and only if
$ \Gamma^3 \equiv_{\rm cyc} \Gamma'^3$.
\end{enumerate}
\end{lemma}
\begin{proof}
It is well known that the
poset of flats of a matroid $M$ depends on, and completely
determines, any {\it simple} matroid $\widetilde{M}$ (see below)
associated to $M$ (see \cite[Sec. 1.7]{Oxl}).
Therefore, using Theorem \ref{matroid}, we will be done if we show that
$\widetilde{M^*(\Gamma)}=M^*(\Gamma^3)$
for any choice of $\Gamma^3$ of $\Gamma$.
Since the cographic matroid does not depend on the choice of the orientation
(cf. Remark~\ref{mator}),
we can fix an orientation on $\Gamma$ inducing a totally cyclic orientation
on $\Gamma \smallsetminus \Esep$, and we let $\Gamma^3$
have the orientation induced by that of $\Gamma$.

Recall (see loc. cit.) that a simple matroid $\widetilde{M^*(\Gamma)}$
is obtained from $M^*(\Gamma)$ by deleting the zero vectors
and, for each parallel (i.e. proportional) class of vectors, deleting all but one of the vectors.
We know $e^*\in H_1(\Gamma,\R)^*$ is zero if and only
$e\in E(\Gamma)_{\rm sep}$ (see \ref{e0sep}).
On the other hand,
Corollary \ref{cH}(\ref{cH2})  yields that $e_1^*$ and $e_2^*$ are proportional
if and only if  they belong to the same C1-set,
if and only if, by Lemma~\ref{C1lm}(\ref{C14}),  $(\{e_1,e_2\})$ is separating pair of edges.
Therefore the edges deleted to pass from
 $M^*(\Gamma)$ to $\widetilde{M^*(\Gamma)}$ correspond  exactly to the edges contracted  to construct $\Gamma^3$ from $\Gamma$, and hence we get that $\widetilde{M^*(\Gamma)}\cong M^*(\Gamma^3)$.
\end{proof}

\subsection{The posets  $\OP_{\Gamma}$ and ${\ov{\OP_{\Gamma}}}$}
 We defined totally cyclic orientations in Definition~\ref{tot}. Now we introduce a partial ordering among them.
\begin{defi}
\label{totdef}
The poset $\OP_{\Gamma}$ of \emph{totally cyclic orientations} of $\Gamma$ is the set of pairs
$(S, \phi_S)$ where $S\in \SP_{\Gamma}$ and
$\phi_S$
%: E(\Gamma\smallsetminus S)\stackrel{(s,t)}{\longrightarrow} V(\Gamma) \times V(\Gamma)$
is a totally cyclic orientation of
$\Gamma\smallsetminus S$, endowed with the following partial order
$$(S, \phi_S)\geq (T,\phi_T) \Leftrightarrow   S\subset  T
\text{ and } \phi_T=(\phi_S)_{|E(\Gamma\smallsetminus  T)}.
$$
We call $S$ the {\it support} of the orientation $\phi_S$.
\end{defi}
We have a natural map
$$\begin{aligned}
 \supp:\OP_{\Gamma} & \rightarrow \SP_{\Gamma}\\
(S, \phi_S)& \mapsto S
\end{aligned}
$$
which is order-preserving by definition and  surjective because
of Lemma \ref{chso}(\ref{c2}).

We say that a map $\pi:(P,\leq)\to (Q,\leq)$ between two posets $(P,\leq)$ and $(Q,\leq)$
is a {\it quotient} if and only if for every $x,y\in Q$ we have that
$$x\leq y \Leftrightarrow   \text{ there exist } \widetilde{x}\in \pi^{-1}(x)
\text{ and } \widetilde{y}\in \pi^{-1}(y) \text{ such that } \widetilde{x}\leq
\widetilde{y}.
$$
In particular $\pi$ is monotone and surjective. Observe also that if
$\pi:(P,\leq)\to (Q,\leq)$ is a quotient, then $(P,\leq)$ is graded if and only if
$(Q,\leq)$ is graded, and in this case we can choose two rank functions $\rho_P$ on $P$
and $\rho_Q$ on $Q$ such that $\rho_Q(\pi(x))=\rho_P(x)$.

We introduce now the outdegree function.

\begin{defi}\label{outdeg}
The outdegree function $\md^+$ is the map
$$\begin{aligned}
\md^+:\OP_{\Gamma}& \longrightarrow \N^{V(\Gamma)}\\
(S, \phi_S) & \mapsto \{d^+(S, \phi_S)_v\}_{v\in V(\Gamma)},
\end{aligned}$$
where $d^+(S, \phi_S)_v$ is the number of edges of $\Gamma\smallsetminus S$ that are
going out of the vertex $v$ according to the orientation $\phi_S$.
\end{defi}

Note that $\md^+$ is monotone with respect to the component-by-component
partial order on $\N^{V(\Gamma)}$.
Moreover
$$
\sum_{v\in V(\Gamma)} d^+(S, \phi_S)_v=\#(E(\Gamma\smallsetminus S)).
$$

This definition enables us to introduce an equivalence relation $\sim$ on $\OP_{\Gamma}$.

\begin{defi}\label{equiv-or}
We say that two elements $(S, \phi_S)$ and $(S', \phi_{S'})$ of $\OP_{\Gamma}$
are equivalent, and we write that $(S, \phi_S)\sim (S', \phi_{S'})$,
if $S=S'$ and
$\md^+(S, \phi_S)=\md^+(S', \phi_{S'})$.
We denote by  $[(S, \phi_S)]$ the equivalence class of $(S, \phi_S)$.

The  set of equivalence classes will be denoted
${\ov{\OP_{\Gamma}}}:=\OP_{\Gamma}/_\sim$.   On  ${\ov{\OP_{\Gamma}}}$ we define
a poset structure by saying that
$[(S, \phi_S)] \geq [(T,\phi_T)]$ if there exist
$ (S', \phi_{S'})\sim (S, \phi)$ and $(T',\phi_{T'})\sim (T,\phi_T)$ such that
$(S', \phi_{S'})\geq (T', \phi_{T'})$ in $\OP_{\Gamma}$.
\end{defi}

Note that ${\ov{\OP_{\Gamma}}}$ is a quotient of the poset $ \OP_{\Gamma}$
and that the natural map of posets $\supp: \OP_{\Gamma}\to \SP_{\Gamma}$
factors as
$$\xymatrix{
\OP_{\Gamma} \ar@{->>}[rr]^{\rm Supp} \ar@{->>}[dr]& & \SP_{\Gamma}  \\
& {\ov{\OP_{\Gamma}}} \ar@{->>}[ru]&
}$$
%In particular, ${\ov{\OP_{\Gamma}}}$ is a graded poset with rank function $[(S, \phi_S)] %\mapsto %b_1(\Gamma\smallsetminus S)$.

The next two lemmas show  that $\OP_{\Gamma}$ and ${\ov{\OP_{\Gamma}}}$ are invariant
under cyclic equivalence and    3-edge connectivization.

\begin{lemma}\label{or-cyclic}
The posets $\OP_{\Gamma}$ and ${\ov{\OP_{\Gamma}}}$ depend only on   $[\Gamma]_{\rm cyc}$.\end{lemma}
\begin{proof}
It is enough to show that the posets $\OP_{\Gamma}$ and
${\ov{\OP_{\Gamma}}}$ do not change under the two moves of Theorem
\ref{cycequ-moves}.

Consider first a move of type (1), that is the gluing of two graphs $\Gamma_1$ and
$\Gamma_2$ at two vertices $v_1\in V(\Gamma_1)$ and $v_2\in V(\Gamma_2)$
(see figure \ref{cont-sep}). Call $\Gamma$ the resulting graph and $v\in V(\Gamma)$ the
resulting vertex.
It is clear that $(\SP_{\Gamma},\leq) \cong (\SP_{\Gamma_1\coprod \Gamma_2},\leq)$.
Given an element $S\in \SP_{\Gamma}$, we denote by
$(S_1, S_2)$ the corresponding element of $\SP_{\Gamma_1\coprod \Gamma_2}$.
It is easy to check that any totally cyclic orientation $\phi_S$ of $\Gamma\smallsetminus S$
induces totally cyclic orientations $\phi_{S_1}$ and $\phi_{S_2}$  of $\Gamma_1\smallsetminus S_1$
and $\Gamma_2\smallsetminus S_2$ and conversely. Moreover the outdegree $\md^+(S,\phi_S)$
determines, and is determined by, the two outdegrees $\md^+(S_1, \phi_{S_1})$ and
$\md^+(S_2, \phi_{S_2})$, hence we get the desired conclusion.

Consider now a move of type (2). Let $\Gamma$ be obtained by gluing the two
graphs $\Gamma_1$ and $\Gamma_2$ according to the rule $u_1\leftrightarrow u_2$
and $v_1\leftrightarrow v_2$, and let $\ov{\Gamma}$ be obtained
by gluing $\Gamma_1$ and $\Gamma_2$ according to the rule $u_1\leftrightarrow v_2$
and $v_1\leftrightarrow u_2$ (see figure \ref{twist}).
Note that since $E(\Gamma)=E(\Gamma_1)\cup E(\Gamma_2)$, any element $S\in \SP_{\Gamma}$
determines two subsets $S_1\subset E(\Gamma_1)$  and $S_2\in E(\Gamma_2)$.
These two subsets $S_1$ and $S_2$ determine also a subset $\ov{S}\in E(\ov{\Gamma})$,
which is easily seen to belong to $\SP_{\ov{\Gamma}}$. The association $S\mapsto \ov{S}$
determines an isomorphism $(\SP_{\Gamma},\leq) \cong (\SP_{\ov{\Gamma}},\leq)$.
We now construct, for any
$S\in \SP_{\Gamma}$, a bijection between
the set of all totally cyclic orientations (resp. totally cyclic orientations
up to equivalence) on $\Gamma\smallsetminus S$ and the set of
 totally cyclic orientations (resp. totally cyclic orientations
up to equivalence) on $\Gamma\smallsetminus {\ov S}$.
Any orientation $\phi_S$ on $\Gamma\smallsetminus S$ determines two orientations
$\phi_{S_1}$ and $\phi_{S_2}$ on $\Gamma_1\smallsetminus S_1$ and $\Gamma_2\smallsetminus S_2$,
respectively. We define an orientation $\phi_{\ov{S}}$ of $\ov{\Gamma}\smallsetminus \ov{S}$
by putting together the orientation $\phi_{S_1}$  and the inverse of the orientation
$\phi_{S_2}$, that is the orientation $\phi_{S_2}^{-1}$ obtained by reversing the
direction of all the edges. Using Lemma \ref{chso}, it is easy to check that
if $\phi_S$ is a totally cyclic orientation of $\Gamma\smallsetminus S$ then $\phi_{\ov{S}}$ is a
totally cyclic orientation of $\ov{\Gamma}\smallsetminus \ov{S}$. Moreover it is straightforward
to check that the outdegree function $\md^+(S, \phi_S)$
determines and is completely determined by  $\md^+(\ov{S},\phi_{\ov{S}})$.
Clearly the association $\phi_S\mapsto \phi_{\ov{S}}$ is a bijection since we can
reconstruct $\phi_S$ starting from $\phi_{\ov{S}}$ by reversing the orientation on
$\Gamma_2$. Moreover it is easy to check that the constructed bijections
$\OP_{\Gamma}\cong \OP_{\ov{\Gamma}}$ and
${\ov{\OP_{\Gamma}}} \cong {\ov{\OP_{\ov{\Gamma}}}} $ are compatible with the poset
structure, and thus we are done.
\end{proof}

\begin{lemma}\label{or-conn}
For any choice of  $\Gamma^3$
we have natural isomorphisms of posets:
$ \OP_{\Gamma} \cong \OP_{\Gamma^3} $ and
${\ov{\OP_{\Gamma}}} \cong  {\ov{\OP_{\Gamma^3}}}$.
\end{lemma}
\begin{proof}
It is enough to show that the posets $\OP_{\Gamma}$ and
${\ov{\OP_{\Gamma}}}$ do not change under the two moves of Definition
\ref{3-conn}.

Recall that for every $S\in \SP_{\Gamma}$ we have $\Esep\subset S$. Therefore
$\Esep$ does not affect the totally cyclic orientations on $\Gamma\smallsetminus S$,
nor does it affect the outdegree function.
This proves that $\OP_{\Gamma}$ does not change when separating edges of $\Gamma$ get contracted.

Consider now a move of type (B), that is the contraction of an edge $e_1$
belonging to a separating pair $(e_1, e_2)$. We refer to the notations
of figure \ref{con-pair}.
We known that $ \SP_{\Gamma}  \cong  \SP_{\ov{\Gamma}}$, by Lemma
\ref{supp-equiv}. Given an element $S\in \SP_{\Gamma}$, we denote by $\ov{S}$ the
corresponding element in $\SP_{\ov{\Gamma}}$. We now construct, for any
$S\in \SP_{\Gamma}$, a bijection between
the set of all totally cyclic orientations (resp. totally cyclic orientations
up to equivalence) on $\Gamma\smallsetminus S$ and the set of
 totally cyclic orientations (resp. totally cyclic orientations
up to equivalence) on $\Gamma\smallsetminus {\ov S}$. If $\ov{e}\in \ov{S}$
(which happens exactly when $e_1, e_2\in S$), then $\ov{\Gamma}\smallsetminus \ov{S}$ is
cyclically equivalent to $\Gamma\smallsetminus S$ and therefore we conclude by the
previous Lemma. If $\ov{e}\not\in \ov{S}$ (which happens exactly when
$e_1$ and $e_2$ do not belong to $S$), we lift any totally cyclic orientation
$\phi_{\ov{S}}$ of $\ov{\Gamma}\smallsetminus \ov{S}$ to an orientation
$\phi_S$ of $\Gamma\smallsetminus S$ by orienting any edge in
$E(\Gamma\smallsetminus S\cup \{e_1\})$ as the corresponding edge in
$E(\ov{\Gamma}\smallsetminus S)$, and by orienting
$e_1$ so that the cycle $\Gamma(\{e_1,e_2\})$ is cyclically
oriented.  Lemma \ref{chso} implies that $\phi_S$ is
a totally cyclic orientation of $\Gamma\smallsetminus S$ and that any totally cyclic orientation
$\phi_S$ must arise from a totally cyclic orientation of $\phi_{\ov{S}}$ via this
construction. Moreover, it is easy to check that the outdegrees
$d^+(S, \phi_S)$ and $d^+(\ov{S},\phi_{\ov{S}})$ are completely determined
one from another, and this concludes the proof.
\end{proof}

\begin{nota}{\emph{A conjectural geometric description of
$\OP_{\Gamma}$ and $\ov{\OP_{\Gamma}}$}}

We propose a conjectural geometric description of the two posets $\OP_{\Gamma}$ and
$\ov{\OP_{\Gamma}}$. Recall the following definition (see for example
\cite[Pag. 174]{BdlHN}).

\begin{defi}\label{Vordef}
The Voronoi polyhedron of the graph $\Gamma$ is the compact convex
polytope defined by
$$\Vor_{\Gamma}:=\{x\in H_1(\Gamma,\R)\: :\: (x,x)\leq (x-\lambda, x-\lambda) \text{ for
all } \lambda\in H_1(\Gamma,\Z)\}. $$
\end{defi}

We denote with $\Fac(\Vor_{\Gamma})$ the poset of faces of the Voronoi polyhedron
$\Vor_{\Gamma}$, with the order given by the reverse of the natural inclusion
between the faces. It is a graded poset with minimum equal to the interior of
$\Vor_{\Gamma}$ and normalized rank function equal to the codimension of the faces.

From the definition, it follows that $\Vor_{\Gamma}$ is a fundamental domain
for the action of $H_1(\Gamma,\Z)$ on $H_1(\Gamma,\R)$ by translations.
In particular $H_1(\Gamma,\Z)$ acts by translation on the faces of $\Vor_{\Gamma}$.
We denote with $\ov{\Fac(\Vor_{\Gamma})}$ the quotient poset
of $\Fac(\Vor_{\Gamma})$ with respect to the action of $H_1(\Gamma,\Z)$.

\begin{conj}\label{geo-conj}
For a graph $\Gamma$, we have that
\begin{enumerate}[(i)]
\item $\OP_{\Gamma}\cong \Fac(\Vor_{\Gamma})$.
\item $\ov{\OP_{\Gamma}}\cong \ov{\Fac(\Vor_{\Gamma})}$.
\end{enumerate}
\end{conj}

The above conjecture $(i)$ generalizes the bijection (proved in \cite[Prop. 5.2]{OS}
and \cite[Prop. 6]{BdlHN}) between the codimension-one faces of $\Vor_{\Gamma}$
and the oriented cycles of $\Gamma$ (which correspond to the elements $S\in
\OP_{\Gamma}$ such that $b_1(\Gamma\smallsetminus S)=1$).
Therefore part $(i)$ proposes an answer to the interesting problem
posed in \cite[Page 174]{BdlHN}:
``More ambitiously, one would like to understand the combinatorics
of the Voronoi polyhedron in terms of oriented circuits of the graph''.

\end{nota}

\subsection{Conclusions}
\begin{lemma}\label{supp-map}
The support map  $\supp:\OP_{\Gamma}   \la \SP_{\Gamma}$
is a quotient of posets. Moreover, given $S, T\in \SP_{\Gamma}$ such that
$S$ covers $T$, and a totally cyclic orientation $\phi_T$ of $\Gamma\smallsetminus T$, there
are at most two (possibly equal) extensions of $\phi_T$ to a totally cyclic orientation $\phi_S$
of $\Gamma\smallsetminus S$.
\end{lemma}
\begin{proof}
We already observed that   $\supp$ is  surjective and order preserving.
For the remaining part   we use the fact that $\SP_{\Gamma}$ is graded
by the function $b_1(\Gamma \smallsetminus S)$(see \ref{rank}).
By \ref{supp-equiv} and \ref{or-conn} we can assume that $\Gamma$ is 3-edge connected.
In particular, we have $\Esep = \emptyset$.

It is easy to see that it suffices to
  assume  $S=\emptyset$. The hypothesis that $\emptyset$ covers $T$ is equivalent
to the fact that $b_1(\Gamma)=b_1(\Gamma\smallsetminus T)+1$
or, equivalently, that  $b_1(\Gamma(T))=1$. Hence
$\Gamma(T)$ is a cycle (as $\Esep = \emptyset$).

Using the characterization
\ref{chso}
(in particular part (\ref{vw})) it is easy to check the only way to extend the orientation $\phi_T$
of $\Gamma\smallsetminus T$ to a totally cyclic orientation on all of $\Gamma$ is by choosing for the edges
of $T$ one of the two cyclic orientations of the cycle $\Gamma(T)$.
\end{proof}

Summing up what we have proved in this section, we get the following

\begin{thm}\label{final-thm}
Let $\Gamma$ and $\Gamma'$ be two
%C connected
graphs.
The following facts are equivalent:
\begin{enumerate}[(i)]
 \item
 \label{(i)}  $ [\Gamma^3]_{\rm cyc} =[\Gamma'^3]_{\rm cyc}$.
\item
 \label{(ii)} ${\rm Del}(\Gamma)\cong {\rm Del}(\Gamma')$.
\item
 \label{(iii)} $\SP_{\Gamma}\cong \SP_{\Gamma'}$ as posets.
\item
 \label{(iv)}$\OP_{\Gamma}\cong \OP_{\Gamma'}$ as posets.
 \item
  \label{(v)} ${\ov{\OP_{\Gamma}}} \cong{\ov{\OP_{\Gamma'}}}$ as posets.
\end{enumerate}
\end{thm}
\begin{proof}
The equivalence (\ref{(i)})$\Leftrightarrow$(\ref{(ii)}) was proved in Proposition
\ref{Del-equ}(\ref{Del3}),
while the equivalence (\ref{(i)})$\Leftrightarrow$(\ref{(iii)}) follows from Lemma~\ref{supp-equiv}.
The implications (\ref{(i)})$\Rightarrow$(\ref{(iv)}) and (\ref{(i)})$\Rightarrow$(\ref{(v)}) follow from Lemmas
\ref{or-cyclic} and \ref{or-conn}. Finally the implications (\ref{(iv)})$\Rightarrow$(\ref{(iii)}) and
(\ref{(v)})$\Rightarrow$(\ref{(iii)}) follows from the fact that $\SP_{\Gamma}$ is a quotient poset
of ${\ov{\OP_{\Gamma}}}$ and $\OP_{\Gamma}$ (see Lemma \ref{supp-map} and the discussion after
definition \ref{equiv-or}).
\end{proof}

\appendix

\section{Tropical Torelli map}\label{trop-torelli}

\subsection{}
Tropical Geometry is a rather young area of mathematics, which has been developing fast
in recent years. Nevertheless an exhaustive construction of the tropical analogues of certain fundamental algebro-geometric objects
is not yet available in the literature; this is the case
of the moduli spaces of curves and  abelian varieties,
in which we are
here interested.

In this appendix, fixing $g\geq 2$ and assuming  
the existence and some  properties of the moduli spaces $\Mgt$ (of compact tropical curves
of genus $g$)  and $\Agt$ (of principally polarized tropical abelian varieties
of dimension $g$), and of the 
  tropical Torelli map
$\tgt: \Mgt\to \Agt$,
we prove    that the map $\tgt$ is of tropical
degree one, as conjectured by  Mikhalkin and Zharkov in
\cite[Sect. 6.4]{MZ}.  
This is done in Theorem~\ref{MZ-conj}, whose proof relies on Theorem  \ref{main-thmt}.
The present appendix greatly benefited from several e-mail exchanges with G. Mikhalkin.

 In this section we assume that all graphs are  connected and all tropical curves are   compact
 and connected.
Recall that a graph is called {\it $3$-regular} if all its vertices
have valence $3$. We say that a tropical curve $C$ is $3$-regular if its associated graph  
  is $3$-regular (see Remark~\ref{termtrop}). 

\begin{remark}
\label{3g-3e}
A connected $3$-regular graph of genus (i.e. first Betti number) $g$ has $3g-3$ edges and $2g-2$ vertices.
\end{remark}
We will need the following easy 

\begin{lemma}\label{3-reg-3-conn}
A $3$-regular graph   is $3$-edge connected if and only if it is $3$-connected.
\end{lemma}
\begin{proof}
Since $3$-connectivity
always implies   $3$-edge-connectivity 
there is 
only one implication to prove. 
Let $\Gamma$ be a $3$-regular graph;
suppose   that $\Gamma$ is not $3$-connected
and let us prove that it is not $3$-edge connected. 

First of all, if $\Gamma$ has a loop based at a vertex $v$
then, as $v$ has valence 3 and the loop contributes by two to the valence 
of $v$, there is a unique other edge attached to $v$, 
which is necessarily a separating edge of $\Gamma$, 
so we are done. We can assume that $\Gamma$ has no loops.

Assume that $\Gamma$ has a separating vertex, $v$.
Let $\Gamma'$ be the graph obtained from $\Gamma$ 
by removing $v$ and all the edges adjacent to it.
By assumption, $\Gamma'$ is not connected. Since $v$ has valence 
$3$ 
there is a connected component $\Gamma''$ 
of $\Gamma'$ with the property that there exists a unique edge of $\Gamma$,
call it $e$, connecting $v$ with one of the vertices of $\Gamma''$.
Clearly $e$ is a separating edge of $\Gamma$, and we are done.
We can assume that $\Gamma$ is $2$-connected. 

Since
$\Gamma$ is not $3$-connected, 
there exists a separating pair of vertices,  $\{v_1, v_2\}$, none of 
which is a separating vertex.
Let $\Gamma'\subset \Gamma$ be the graph obtained from $\Gamma$ by removing $v_1$, $v_2$, and all the edges
adjacent to them. By assumption $\Gamma '$ has at least two connected components.
Since $v_i$ has valence 3 for $i=1,2$, there are at most two edges joining
 $v_1$ and $v_2$ 
(for otherwise $\Gamma$ has no other vertex and there is nothing to prove).
Therefore
there exists an  edge, $e_i$,  touching $v_i$ and such that
the component of  $\Gamma '$  
 which touches $e_i$   touches no other edge adjacent to $v_i$.
  It is clear that the pair $\{e_1, e_2\}$ is a disconnecting pair of edges of $\Gamma$.
Notice that   the edge $e_i$ is not necessarily unique, but we are free to make a   choice without changing
the  conclusion.
 The proof is  complete.

\end{proof}
We shall now make some reasonable assumptions on 
  the moduli spaces $\Mgt$, $\Agt$ and on the tropical map $\tgt$.
  After each set of assumptions we will provide  motivations and, when possible, some references.

\begin{assumption}\label{Mg-ass}
There exists a moduli space $\Mgt$ of dimension $3g-3$
para\-metrizing tropical equivalence classes of (compact) tropical curves
of genus $g$. A generic point of $\Mgt$ is a $3$-regular tropical curve.
\end{assumption}

For a strategy to construct $\Mgt$, see \cite[Sect. 3.1]{MIK0}.

A neighborhood of a $3$-regular curve $C$ in $\Mgt$ is obtained by varying the lengths $l(e)$
of the edges of the corresponding graph, and therefore, by Remark  \ref{3g-3e}, it is isomorphic to  
$\R_{>0}^{3g-3}\subset \R^{3g-3}$. This explains the dimension requirement.

The condition on the generic point (i.e. a point varying in an open dense subset of $\Mgt$)
derives from the fact that  specializations of a tropical curve
are obtained by letting some of its edge lengths  go to $0$, i.e. by contracting
some of its edges (see \cite[Sec.3.1.D]{MIK0}).
Therefore, if we have a tropical curve $C_0$ with a vertex $v$ of
valence $k+l\geq 4$, with $k, l\geq 2$, we can realize it as the limit of  tropical curves $C_t$ in which the vertex
$v$ is replaced by two vertices $v_1$ and $v_2$ of valence resp. $k+1$ and $l+1$  joined
by a new edge $e$:
\begin{figure}[!htp]
$$\xymatrix@=1pc{
&&*{\bullet} \ar@{.}[dd]^{k} \ar@{-}[drr]&& & && *{\bullet} \ar@{.}[dd]_{l} &&&
*{\bullet} \ar@{.}[dd]^{k} \ar@{-}[drr]&& &&*{\bullet} \ar@{.}[dd]_{l} &&\\
C_t\ar@{=}[r]&&&&*{\bullet}\ar@{-}[r]^{e}&*{\bullet} \ar@{-}[urr]\ar@{-}[drr]_(0.1){v_2}&&&
\ar@{->}[r]&& && *{\bullet}\ar@{-}[urr]\ar@{-}[drr] & && \ar@{=}[r]& C_0\\
&&*{\bullet} \ar@{-}[urr]_(.9){v_1}&& & && *{\bullet} &&&
*{\bullet} \ar@{-}[urr]_(.9){v}&& &&*{\bullet}&&
}$$
\caption{$C_t$ specializes to $C_0$ by letting $l(e)\to 0$.}
\label{trop-spe}
\end{figure}
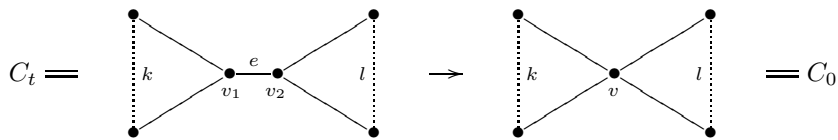

By applying the above procedure on all the vertices of valence greater than $3$,
we obtain that every tropical curve is the limit of
$3$-regular tropical curves. 
%Conversely, if a tropical curve
%$C'$ is obtained by specialization of another tropical curve $C$, then $C'$ must
%contain at least one vertex of valence greater than $3$, and therefore it cannot be
%$3$-regular.

The previous discussion yields the following useful observation.
\begin{remark}
\label{sptrop}
Let $C_0$ be a tropical curve and $(\Gamma_0,l_0)$ the associated metric graph.
To prove that $C_0$ is the specialization of tropical  curves $C_t$ of a certain combinatorial type,
i.e. having a certain
associated graph $\Gamma_t$, 
is equivalent
%L%
to proving that there exists a surjective edge-contracting map
$\sigma:\Gamma_t \to \Gamma_0$, such that $b_1(\Gamma_t)=b_1(\Gamma_0 )$.

Indeed, it suffices to define the length function $l_t$ for $C_t$ so that
$l_t$ equals $l$ on the edges not contracted by $\sigma$, and $l_t$ tends to $0$ on every edge contracted by $\sigma$.
\end{remark}

\begin{assumption}\label{Ag-ass}
There exists a moduli space $\Agt$ of dimension $\binom{g+1}{2}$
parame\-trizing principally polarized tropical abelian
varieties of dimension $g$.
\end{assumption}

\begin{assumption}\label{tg-ass}
The Torelli map
$$\begin{aligned}
\tgt:\Mgt& \longrightarrow \Agt\\
C & \mapsto (\Jac(C), \Theta_C)
  \end{aligned}
$$
is a tropical map.
Its image will be denoted  $\Jgt:=\tgt(\Mgt)$.
%L% called the Jacobian locus  

Locally around a $3$-regular tropical curve $C$,
$\tgt$ is given by the restriction of a linear map $L_C: \R^{3g-3}\to \R^{\binom{g+1}{2}}$
which sends $\Z^{3g-3}$ to $\Z^{\binom{g+1}{2}}$.
\end{assumption}

Note that $\Jgt$  is a tropical variety.
A neighborhood in $\Agt$ of $\tgt(C)$, for some $C\in \Mgt$,
is obtained by varying the entries of the symmetric $g\times g$ matrix associated to the scalar
product $(,)_l$ on $H_1(C, \R)$, and therefore it is isomorphic to an open subset of $\R^{\binom{g+1}{2}}$.
 Since the entries of $(,)_l$ depend linearly on the length
function $l$, we get that $\tgt$ is given locally around $C$ by the restriction of a linear map
$L_C:\R^{3g-3}\to \R^{\binom{g+1}{2}}$. If the length function $l$ takes values
in $\Z$ then also the entries of $(,)_l$ with respect to a basis in $H_1(C,\Z)$
will be in $\Z$.
%L% and therefore $L_C$ preserves the canonical integral structures.

Mikhalkin and Zharkov conjectured (in \cite[Sec. 6.4]{MZ}) that $\tgt$, although
not injective, is of tropical degree one to its image,
in the following sense:
\begin{defi}\label{tr-deg1}
The map $\tgt:\Mgt\to \Jgt$ is of {\it tropical degree one}
 if the following two conditions are satisfied:
\begin{enumerate}[(i)]
 \item\label{con-deg1} 
 The inverse image via $\tgt$
 of a generic point $\tgt(C)$ of $\Jgt$
 consists only of $C$.
\item \label{con-deg2} For a generic point $\tgt(C)$ of $\Jgt$, the linear map
$L_C$ of Assumption~\ref{tg-ass} is {\it primitive}, i. e.
$L_C^{-1}\left(\Z^{\binom{g+1}{2}}\right)=\Z^{3g-3}$.
\end{enumerate}
\end{defi}

\subsection{}
In this final part we shall prove the following
\begin{thm}\label{MZ-conj}
The Torelli map $\tgt:\Mgt\to \Jgt$ is of tropical
degree one.
\end{thm}
The proof  will be at the end, combining
 Theorem~\ref{main-thmt} with   Proposition~\ref{spec3}.  
\begin{nota}
\label{blowup}
Let $v\in V(\Gamma)$; we denote by $E_v(\Gamma)$ the set of edges of $\Gamma$ that are 
adjacent to the vertex $v$.
%%r<
Suppose that $v$ has valence $M\geq 4$. In the proof of the following proposition we will use  an  elementary operation, called  a {\it (valence reducing) extension} of $\Gamma$ at $v$,
which is a kind of inverse     to the   contraction of an edge in $v$, and which has the effect of replacing $v$ by two new vertices of valence at least 3 and strictly smaller than $M$.
%%r>
This will be a (not unique) graph $\Gamma'$ endowed with a map $\Gamma '\to \Gamma$
which is an isomorphism away from $v$ and which contracts a unique non loop edge, $e'$, to $v$.
Thus $\Gamma '$ is obtained by replacing the vertex $v$ by an edge  $e'$  having   extremal vertices 
$u'$ and $ w'$, and by distributing between $u'$ and $w'$ the edges adjacent to $v$, so that both $u'$ and $w'$ have valence at least 3.
More precisely,
we denote $V(\Gamma)=\{v, v_2\ldots v_{\gamma}\}$ and $V(\Gamma')=\{u',w', v_2'\ldots v_{\gamma}'\}$ with $v_i$ corresponding to $v_i'$ via the contraction map; similarly
we denote $E(\Gamma)=\{e_1\ldots e_{\delta}\}$ and $E(\Gamma')=\{e',e_1'\ldots e_{\delta}'\}$.
The edge $e'$ joins $u'$ to $w'$;
if an edge $e_i$ joins $v_j$ to $v_i$, then $e_i'$ joins $v_j'$ to $v_i'$. If an edge $e_i$ joins $v$ to $v_i$ we need to specify
whether $e'_i$ joins $u'$ to $v_i'$, or  $w'$ to $v_i'$ (which is why $\Gamma'$ is not unique).
For whichever choice we make, provided that $u'$ and $w'$ have valence at least 3 (this is possible as $M\geq 4$,
and there is   $e'$ joining $u'$ with $w'$)
it is clear that there is a map $\Gamma '\to \Gamma$ as wanted.

\begin{remark} 
\label{3'3}
Let $\Gamma$ be $3$-edge conected, and
let $\Gamma '\to \Gamma$ be an extension of $\Gamma$ at  a vertex $v$, such that
$e'$ is not a separating edge of $\Gamma'$.
Then any separating pair of edges of $\Gamma'$ is of type
$(e', e'_i)$ with $e_i$ a separating edge of $\Gamma\smallsetminus v$.

Here and in the sequel  $\Gamma\smallsetminus v=\Gamma\smallsetminus \{v\}$  
(we omit the braces to ease the notation) denotes the topological space 
obtained by removing the point corresponding to the vertex $v$
from the topological space naturally associated to $\Gamma$. In particular,
 $\Gamma\smallsetminus v$ is not a graph, but we will extend to it 
the terminology used for graphs, as no confusion is likely to arise.
\end{remark}
\end{nota}

\begin{prop}
\label{spec3} The locus  in $\Mgt$ of $3$-edge connected curves is equal to the closure
of the locus of $3$-regular, $3$-connected curves.
\end{prop}
\begin{proof}
Recall that, by  Lemma~\ref{3-reg-3-conn}, a 
$3$-regular, $3$-edge connected graph is also $3$-connected.
Therefore,
by remark ~\ref{sptrop} it is enough to show the following. 
There exists a $3$-regular, $3$-edge connected graph  $\Gamma^*$
together with a surjective edge-contracting  map $\Gamma^* \to \Gamma$
such that $b_1(\Gamma)=b_1(\Gamma ^*)$ if and only if $\Gamma$ is $3$-edge connected.

The fact that
if $\Gamma^*$ is $3$-regular and  $3$-edge connected then $\Gamma$ is 
  $3$-edge connected   follows  from the easy fact that the edge connectivity 
cannot decrease by edge contraction.

Conversely, consider a $3$-edge connected graph $\Gamma$;
note  that  $\Gamma$ has valence at least $3$.
Let $M$ be the maximum valence of a vertex of $\Gamma$;  
if $M=3$ there is nothing to prove, so suppose $M\geq 4$. We shall prove that a $\Gamma ^*$ as above exists by 
constructing a $3$-edge connected
graph $\Gamma '$ which is an extension of   $\Gamma$,
and such that the number of $M$-valent vertices of $\Gamma'$   is  less than
that of $\Gamma$. It is clear that this will be enough.
Let $v$ be a vertex of valence $M$.

{\bf Step 1.}
Suppose that $v$ is the base of some  loop $e_{\ell}$ of $\Gamma$.
Call $G$ the complement in $\Gamma$ of $v$ and of all loops based at $v$ 

$$
G:=\Gamma \smallsetminus v\smallsetminus
 \bigcup_{e_{\ell} \text{ loop based at } v} e_{\ell}.
$$
If $G$ is not empty there are at least 3 edges
$e_1, e_2, e_3$ in  $E_v(\Gamma)$  contained in $G$ (as $\Gamma$ is $3$-edge connected,).
We apply an extension of  $\Gamma$ at $v$ so that  for every loop $e_{\ell}$ the corresponding
$e_{\ell}'$ joins $u'$ with $w'$ 
(notation in \ref{blowup}), $e_1'$ is adjacent to $u'$ and $e_2', e_3'$ are adjacent to $w'$.
If there are other edges $e_4,\ldots, e_m\in E_v(\Gamma)$,  we distribute
$e_4',\ldots, e_m'$
 between  $E_{u'}(\Gamma ')$  
and $E_{w'}(\Gamma ')$ arbitrarily.
It is clear that $u'$ and $w'$ have valence 
at most $M-1$.
Now notice that both $e'$ and $e_{\ell}'$
join $u'$ with $w'$. Therefore 
  $e'$ is not a separating edge of $\Gamma '$, and  there exists no separating pair of edges
  $(e', e'_i)$ unless $e_i'= e_{\ell}'$ for some loop $ e_{\ell}$ based at $v$. It is obvious that $ e_{\ell}$ is not a separating edge of $\Gamma \smallsetminus v$; by Remark~\ref{3'3} we have that $\Gamma '$ is $3$-edge connected,
  so we are done.

Repeating Step 1 we reduce to the case when $v$ is not the base of any loop.

{\bf Step 2.}
Consider $\Gamma \smallsetminus v$
and suppose it is not connected, i.e. $v$ is a separating vertex of $\Gamma$. Then
 $\Gamma \smallsetminus v$ contains no separating edges.
 Call $G_1,\ldots G_c$  the connected components of $\Gamma \smallsetminus v$.
 Since $\Gamma$ is $3$-edge connected every $G_i$ contains at least 3 edges
 adjacent to $v$; call them $e_1^i, e_2^i, e_3^i$.
 Now apply an extension of  $\Gamma $ at $v$ so that
for every odd $i\leq c$ (respectively for every even $i\leq c$)
 $e_1^i$ is adjacent to $u'$ (resp. to $w'$) and $e_2^i, e_3^i$ 
 are adjacent to $w'$ (resp. to $u'$). 
In this way $e'$ is not a separating edge of $\Gamma'$,
 and hence $\Gamma'$ is $3$-edge connected (by Remark~\ref{3'3}).
By distributing the remaining (if any)
 edges coming from $E_v(\Gamma)$ between $w'$ and $u'$ 
 we have that, since $c\geq 2$,  $u'$ and $w'$ have smaller valence than $v$.
So we are done.

 %%r<
{\bf Step 3.} Now suppose $\Gamma \smallsetminus v$ connected and free from separating edges.
Let $\Gamma'$ be a valence reducing  extension of $\Gamma$ at $v$ such that   both $u'$ and $w'$ are at least 3-valent.
Now $e'$ is not  a separating edge of $\Gamma '$ ($v$ is not a separting vertex of $\Gamma$) and hence
by Remark~\ref{3'3}, $\Gamma'$ ie  $3$-edge connected, as wanted..
%%r>

{\bf Step 4.}
Finally, suppose  $\Gamma \smallsetminus v$ connected and admitting a non empty set,
 $S_v\subset E(\Gamma)$, of  
 separating edges. 
 An example of the construction we are going to illustrate is in Figure~\ref{blfig}.
 Note that no edge in $S_v$ is adjacent to $v$, for otherwise  $\Gamma \smallsetminus v$
would be disconnected. 
Also, for any extension of $\Gamma$ at $v$, the new edge $e'$ is not separating.

Consider the graph $\widetilde{\Gamma}:=\Gamma \smallsetminus S_v$.
The valence of $v$ as a  vertex of $\widetilde{\Gamma}$
is  again   $M\geq 4$.
If $\widetilde{\Gamma}$ fails to be connected, we replace it by the connected component containing $v$, which will not affect the conclusions.

It is clear that describing an extension  $\Gamma '$ of $\Gamma$ at $v$  is equivalent to describing
an extension $\widetilde{\Gamma}'$ of  $\widetilde{\Gamma}$ at $v$ (one just needs to glue back to
$\widetilde{\Gamma}'$ the edges in $S_v$ and the connected components 
of $\Gamma \smallsetminus S_v$ not containing $v$).
Now, for $\Gamma'$ to be $3$-edge connected we need to extend   $\widetilde{\Gamma}$  so that, for any $e_s\in S_v$,
the pair $(e',e_s')$ is not a separating pair of $\Gamma '$,
i.e. (since $e'$ is not a separating edge of $\Gamma'$)
so that $e'$ is not a separating edge of $\Gamma '\smallsetminus e_s'$.

The space  $\widetilde{\Gamma}\smallsetminus v$ is not connected 
and we proceed like for Step 2,   notice however that
 $\widetilde{\Gamma}$ is only $2$-edge connected.
 %%%
Call $H_1,\ldots H_h$  the connected components of $\widetilde{\Gamma}\smallsetminus v$.
 As   $\widetilde{\Gamma}$ is $2$-edge connected each $H_i$ contains at least 2 edges
 adjacent to $v$; call them $e_1^i, e_2^i$.
 Now apply an extension of  $\widetilde{\Gamma}$ at $v$ so that
for every odd $i$ (resp. for every even $i$)
 $e_1^i$ is adjacent to $u'$ (resp. to $w'$) and $e_2^i$ 
is adjacent to $w'$ (resp. to $u'$). 
 By doing this,
$e'$ is not a separating edge of  $\widetilde{\Gamma}'=\Gamma'\setminus S_v$,
and hence a fortiori $e'$ is not a separating edge of $\Gamma '\smallsetminus e_s'$.
As we observed above, this implies that $\Gamma'$ is $3$-edge connected.
By distributing between $w'$ and $u'$ the  remaining 
 edges  
  we have that, as $h\geq 2$, the vertices $u'$ and $w'$ have smaller valence than $v$.

Since the removal of $S_v$, or of components of
$\Gamma \smallsetminus S_v$  not containing $v$, does not alter $\Gamma$ at $v$,
 we are done.
\end{proof}
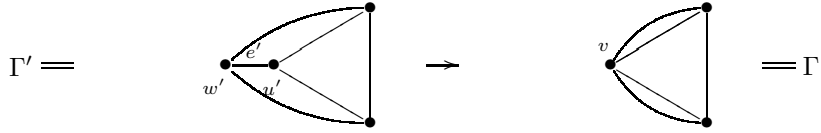
\begin{figure}[!htp]
$$\xymatrix@=1pc{
&&&& & && *{\bullet} \ar@{-}[dd] \ar@{-}@/_/[dlll]  &&&
&& &&*{\bullet} \ar@{-}[dd]\ar@{-}@/_/[dll]_(.9){v}&&\\
\Gamma '\ar@{=}[r]&&&&*{\bullet}\ar@{-}[r]^{\  e'}
\ar@{-}@/_/[drrr]_(0.05){w'}  &*{\bullet} \ar@{-}[urr]\ar@{-}[drr]_(0.1){u'}&&&
\ar@{->}[r]&& && *{\bullet}\ar@{-}@/_/[drr] \ar@{-}[urr]\ar@{-}[urr]\ar@{-}[drr] & && \ar@{=}[r]& \Gamma \\
&&&& & && *{\bullet} &&&
&& &&*{\bullet}&&
}$$
\caption{$3$-edge connected extension of $\Gamma$ at $v$.}
\label{blfig}
\end{figure}

\noindent {\it Proof of Theorem~\ref{MZ-conj}.}
We claim that the
generic point of $\Jgt$ is of the form $\tgt(C)$ for a $3$-regular and $3$-connected curve
$C$.
It is clear that the generic point of  $\Jgt$ is the image of some generic point of $\Mgt$,
i.e. of some $3$-regular curve.
Let $C$ be a $3$-regular curve not necessarily $3$-connected.
By Theorem \ref{main-thmt} we have $\tgt(C)=\tgt (C^3)$, where $C^3$
is the 3-edge connected curve associated to a $3$-edge connectivization 
(see \ref{3-connl}) of the metric graph of $C$.
By Proposition~\ref{spec3}, $C^3$ is the specialization of a $3$-regular,
$3$-connected curve.
Therefore $\tgt(C^3)=\tgt (C)$ is the specialization of the Torelli image of a 
$3$-regular $3$-connected curve.
This proves the claim.

Now, for a $3$-connected curve $C$, 
Theorem \ref{main-thmt} implies that $(\tgt)^{-1}$ $(\tgt(C))=\{C\}$ and therefore 
 condition (\ref{con-deg1})  of Definition
 \ref{tr-deg1} is satisfied.

Consider now the linear map $L_C$ of Assumption \ref{tg-ass}. 
By what we just proved we can take $C$ a $3$-connected, $3$-regular curve.
To prove that $L_C$
is primitive, we have to show that if the scalar product $(,)_l$ on $H_1(C, \R)$ takes integer
values on $H_1(C, \Z)$ then the length function $l$ takes also integer values.

Recall that, by  Corollary~\ref{cor3}, to say that $C$ is $3$-edge connected is to say that
for every edge $e$
of the graph $\Gamma$ associated to $C$ the set $\{e\}$ is a C1-set of $\Gamma$.
Hence, using Lemma \ref{intS}, we deduce that for every   $e$
 there exist two cycles, $\Delta_1$ and $\Delta_2$, of $\Gamma$ such that $\{e\}=E(\Delta_1)\cap E(\Delta_2)$. Therefore there exist two elements $c_1, c_2\in H_1(C,\Z)$
such that $l(e)=(c_1, c_2)_l$. Since  by assumption $(,)_l$ takes integer values on
$H_1(C, \Z)$ we get that $l(e)=(c_1,c_2)_l \in \Z$ for every edge $e$ of $\Gamma$.

The proof of Theorem~\ref{MZ-conj} is complete.
$\qed$

\end{document}